\documentclass[letterpaper]{amsart}
\usepackage[foot]{amsaddr}

\usepackage{amssymb,amsfonts,amsmath,amsthm,mathtools}
\usepackage[all,arc]{xy}
\usepackage{xcolor}
\usepackage{enumerate}
\usepackage{mathrsfs}
\newtheorem{theorem}{Theorem}[section]
\newtheorem*{theorem*}{Theorem}
\newtheorem{corollary}[theorem]{Corollary}
\newtheorem{proposition}[theorem]{Proposition}
\newtheorem{lemma}[theorem]{Lemma}

\theoremstyle{definition}
\newtheorem{definition}[theorem]{Definition}

\theoremstyle{remark}
\newtheorem{remark}[theorem]{Remark}

\DeclareMathOperator{\cC}{\mathcal{C}}
\DeclareMathOperator{\cP}{\mathcal{P}}
\DeclareMathOperator{\cZ}{\mathcal{Z}}
\DeclareMathOperator{\cS}{\mathcal{S}}

\DeclareMathOperator{\cF}{\mathcal{F}}
\DeclareMathOperator{\cK}{\mathcal{K}}
\DeclareMathOperator{\cM}{\mathcal{M}}

\DeclareMathOperator{\cV}{\mathcal{V}}

\DeclareMathOperator{\bI}{\mathbf{I}}
\DeclareMathOperator{\bC}{\mathbf{C}}
\DeclareMathOperator{\bH}{\mathbf{H}}
\DeclareMathOperator{\bF}{\mathbf{F}}

\DeclareMathOperator{\bM}{\mathbf{M}}
\DeclareMathOperator{\bL}{\mathbf{L}}

\DeclareMathOperator{\R}{\mathbb{R}}
\DeclareMathOperator{\RP}{\mathbb{RP}}
\DeclareMathOperator{\Z}{\mathbb{Z}}
\DeclareMathOperator{\N}{\mathbb{N}}

\DeclareMathOperator{\area}{area}
\DeclareMathOperator{\spt}{spt}
\DeclareMathOperator{\dmn}{dmn}

\DeclareMathOperator{\vol}{vol}
\DeclareMathOperator{\Int}{int}
\DeclareMathOperator{\ind}{index}
\DeclareMathOperator{\Ric}{Ric}
\DeclareMathOperator{\Ncat}{\mathcal{N}_1-cat}
\DeclareMathOperator{\cat}{cat}
\DeclareMathOperator{\Span}{span}

\makeatletter
\let\c@equation\c@thm
\makeatother

\title[Arbitrarily Large Area]{Existence of minimal hypersurfaces with arbitrarily large area and possible obstructions}

\author{James Stevens}

\author{Ao Sun}

\address{University of Chicago, Department of Mathematics, 5734 S University Ave, Chicago IL, 60637, USA; Current address for A.S: Department of Mathematics, Lehigh University, Chandler-Ullmann Hall, Bethlehem, PA 18015, USA}
\email{jamesstevens@uchicago.edu; aos223@lehigh.edu}

\begin{document}
		\begin{abstract}
		We prove that in a closed Riemannian manifold with dimension between $3$ and $7$, either there are minimal hypersurfaces with arbitrarily large area, or there exist uncountably many stable minimal hypersurfaces. Moreover, the latter case has a very pathological Cantor set structure which does not show up in certain manifolds. Among the applications, we prove that there exist minimal hypersurfaces with arbitrarily large area in analytic manifolds. In the proof, we use the Almgren-Pitts min-max theory proposed by Marques-Neves, the ideas developed by Song in his proof of Yau's conjecture, and the resolution of the generic multiplicity-one conjecture by Zhou.
		
		
	\end{abstract}
	\maketitle

	\tableofcontents
	
	 \section{Introduction}
	
	Minimal surfaces are critical points of the area functional. They are the fundamental geometric models in the calculus of variations, and they play important roles in the study of geometry, topology, and general relativity. A central study is that of the existence of minimal hypersurfaces in a closed manifold. Yau \cite{Ya} conjectured that there exist infinitely many minimal hypersurfaces in a closed manifold. When the ambient manifold has dimension between $3$ and $7$, this conjecture was recently solved by Song \cite{So}, building on the work of Marques-Neves \cite{MN1}, using the novel Almgren-Pitts min-max theory.
	
	The main purpose of this paper is to provide further delicate information on the space of minimal hypersurfaces. Our main theorem shows that a manifold either admits minimal hypersurfaces with arbitrarily large area, or it carries some very pathological metric. Throughout, $(M^{n+1},g)$ is a closed Riemannian manifold with $3 \leq n+1 \leq 7$, and all minimal hypersurfaces considered will be smooth, closed, and embedded. The dimension restriction is due to the fact that min-max minimal hypersurfaces have nice regularity in these dimensions by Schoen-Simon \cite{SS} and Pitts \cite{Pi}. One version of our main result is the following.
	
	\begin{theorem}\label{thm:main}
		If $(M^{n+1},g)$ is a Riemannian manifold with $3 \leq n+1 \leq 7$, either:
		\begin{itemize}
			\item there exist connected minimal hypersurfaces of arbitrarily large area; or
			\item there exist uncountably many connected stable minimal hypersurfaces of uniformly bounded area and with infinitely many distinct areas.
		\end{itemize}
	\end{theorem}
	
	Note that the second case in the above theorem is very pathological, and should not show up under certain conditions, for example, when the manifold is analytic.
	
	\begin{theorem}\label{thm:analytic}
		If $(M^{n+1},g)$ is analytic with $3 \leq n+1 \leq 7$, then there exist connected minimal hypersurfaces of arbitrarily large area.
	\end{theorem}

	In order to state the precise version of our main theorem, we need to introduce a new definition.

	\begin{definition}\label{def:non-monotonic}
		We say a minimal hypersurface $\Gamma \subset (M,g)$ is \emph{non-monotonic} if there exists no $\delta > 0$ along with $\varphi: \Gamma \times [-\delta,\delta]$ such that
		\begin{itemize}
			\item $\varphi(\Gamma \times \{0\}) = \Gamma$,
			\item the mean curvature vector of $\varphi(\Gamma \times \{t\})$ is either everywhere $0$, or non-vanishing, 
			\item $\varphi$ is a diffeomorphism onto its image, and
			\item $\area_g(\varphi(\Gamma \times \{t\}))$ is weakly monotonic on $[-\delta,0]$ and $[0,\delta]$.
		\end{itemize} 
	\end{definition}
	
	 Non-monotonicity is a very pathological behavior of minimal hypersurfaces: the area functional oscillates around a non-monotonic minimal hypersurface. Our main theorem asserts that the only possible obstruction to the existence of minimal hypersurfaces with arbitrarily large area is the space of non-monotonic minimal hypersurfaces is homeomorphic to a Cantor set.
	
	\begin{theorem}\label{thm:main-full}
		Given a closed Riemannian manifold $(M^{n+1},g)$ with $3 \leq n+1 \leq 7$, either:
		\begin{itemize}
			\item there exist connected minimal hypersurfaces of arbitrarily large area; or
			\item the space of non-monotonic accumulating minimal hypersurfaces is homeomorphic to a Cantor set.
		\end{itemize}
	\end{theorem}

In the rest of this introduction, we give some background information, and outline the proof.
	
	\subsection*{Min-max theory}
	Even before Yau's conjecture, Almgren \cite{Al1, Al2} and Pitts \cite{Pi} developed the min-max method to construct at least one closed embedded minimal hypersurface in a Riemannian manifold $(M,g)$. The method is known as the Almgren-Pitts theory. The progress from one to infinitely many was first made by Marques-Neves. In \cite{MN1}, they proved the existence of infinitely many closed embedded minimal hypersurfaces in $(M^{n+1},g)$ with $3\leq n+1\leq 7$ and which satisfies the Frankel property: any two closed embedded minimal hypersurfaces intersect. In particular, manifolds with positive Ricci curvature satisfy the Frankel property. Later, Irie-Marques-Neves \cite{IMN} showed that in a manifold $(M^{n+1},g)$ for $3\leq n+1\leq 7$ with a generic (bumpy) metric, the union of closed embedded minimal hypersurfaces is actually dense in the manifold. Later this result was quantified by Marques-Neves-Song \cite{MNS}. As a consequence, Yau's conjecture for generic metrics was settled.
	
	An important ingredient in \cite{MN1} are estimates of the volume spectrum first shown by Gromov \cite{Gr1} and Guth \cite{Gu}, which show that the min-max widths grow sublinearly. In fact, Gromov \cite{Gr2} conjectured a more precise growth rate called the Weyl law which was later proved by Marques-Neves-Liokumovich \cite{LMN} and played an important role in both \cite{IMN} and \cite{MNS}.
	
	Another approach to Yau's conjecture for generic metrics is proving the multiplicity-one conjecture for Almgren-Pitts min-max theory. This approach was first proposed by Marques-Neves in a series of papers \cite{Ma, MN4, Ne}. In \cite{MN3}, they discussed many nice properties that a multiplicity-one min-max minimal hypersurface would satisfy. The generic multiplicity-one conjecture was recently settled by Zhou in \cite{Zh}, using the ideas from the construction of prescribed mean curvature hypersurfaces by Zhou-Zhu in \cite{ZZ1, ZZ2}. The generic metrics considered in the multiplicity-one conjecture are the bumpy metrics. Recall that a metric is called bumpy if any immersed closed minimal hypersurface is non-degenerate, namely it has no non-trivial Jacobi field. White \cite{Wh} showed that bumpy metrics are generic. This approach was also studied in the Allen-Cahn min-max theory, see \cite{G, GG1, GG2}. The generic multiplicity-one conjecture for Allen-Cahn min-max theory was proved by Chodosh-Mantoulidis \cite{ChM1} for $3$-dimensional manifolds, building on the work of Wang-Wei in \cite{WW}. 
	
	For a general metric which may not be bumpy, both the denseness argument and the multiplicity-one argument fail. In \cite{So}, Song came up with a novel approach which proved the existence of infinitely many closed embedded minimal hypersurfaces in a closed manifold $(M^{n+1},g)$ with an arbitrary metric with $3\leq n+1\leq 7$. This settled Yau's conjecture in full generality. Song introduced a core decomposition and studied the volume spectrum of manifolds with cylindrical ends. These ideas are adapted in our work.
	
	Min-max theory can also be used to construct free boundary minimal hypersurfaces in a manifold with boundary, see \cite{LZ, GMWZ, Wa, SWZ}. We believe that our ideas can also be used to study free boundary minimal hypersurfaces.

	\subsection*{Large area minimal surfaces}
	Even though there exist infinitely many closed embedded minimal hypersurfaces in a closed manifold, it is not clear whether they can have arbitrarily large area. For example, this does not necessarily happen for embedded geodesics in a surface (which we can think as the case $n=2$). Moreover, for a general elliptic variational problem, the critical points may not have arbitrarily large values. For example, in Yang-Mills theory, all the self-dual connections over a  compact oriented $4$-manifold are minimizers to the Yang-Mills functional, and have a fixed value of Yang-Mills functional, see \cite{La}. This fixed value depends on the topology of the $4$-manifold.
	
	From a Morse theoretic point of view, the value of the elliptic functional can reflect the topology of the ambient space. Thus, the existence of minimal hypersurfaces with arbitrarily large area may reflect the topology of the space of minimal hypersurfaces in a closed manifold.
	
	Because of the Weyl law, the minimal hypersurfaces generated by min-max theory will have mass tending to infinity. However, these hypersurfaces coming from min-max theory can have multiple connected components where each component could appear with some multiplicity. But when the ambient metric is bumpy, Chodosh-Mantoulidis \cite{ChM2} utilized the multiplicity-one conjecture to prove there exist connected minimal hypersurfaces with arbitrarily large area. Although a bumpy metric is generic and has many nice properties, we cannot write down any explicit examples of such metrics because checking that a metric is bumpy requires one to know information about all minimal hypersurfaces in a manifold.
	
	Another important quantity studied in min-max theory is the Morse index. Recall that the index of a minimal hypersurface is the number of negative eigenvalues of the linearized operator, with multiplicity counted. Stable minimal hypersurfaces have index $0$. Large area minimal hypersurfaces may not always have large index. For example, Colding-Minicozzi \cite{CoM} constructed embedded stable minimal tori with arbitrarily large area in a closed $3$-manifold. On the other hand, in \cite{Li}, Li used Song's min-max theory and generalized Chodosh-Mantoulidis's work to show that in a manifold $M^{n+1}$ with a bumpy metric and $3\leq n+1\leq 7$, there exist minimal hypersurfaces with both arbitrarily large area and index. 
	
	When $M$ is a $3$-dimensional manifold with positive scalar curvature, Chodosh-Ketover-Maximo \cite{CKM} proved that the area of a closed embedded minimal surface in $M$ is bounded from above in terms of its index. With this result, Theorem \ref{thm:analytic} implies the following consequence.
	
	\begin{corollary}\label{cor:analytic+PSC}
		Given a closed analytic Riemannian manifold $(M^3,g)$ with positive scalar curvature, then there exist connected minimal surfaces of arbitrarily large area and index.
	\end{corollary}
	
	Also, as a byproduct of our proof, we have the following corollary.
	
	\begin{corollary}\label{cor:foliation+PSC}
		If $(M^3,g)$ has positive scalar curvature and is foliated by minimal surfaces, there exist connected minimal surfaces of arbitrarily large area and index.
	\end{corollary}

For instance, one application of this is the following.

\begin{corollary}
	The following manifolds admit minimal surfaces with arbitrarily large area and genus: 
	\begin{itemize}
		\item $S^2\times S^1$ with the product metric and $S^2$ carries the standard sphere metric;
		\item More generally, $S^2\times S^1$ with the product metric and $S^2$ carries any smooth metric with positive curvature.
	\end{itemize}
	
\end{corollary}
	
	\subsection*{Outline of proof}
	The main idea is trying to cut the manifold into small pieces. This idea has been used by Song in \cite{So}. Nevertheless, Song's proof of Yau's conjecture used a contradictory argument so that the cutting process only involves in finitely many minimal hypersurfaces. In our case, the cutting process can get complicated since there may be infinitely many minimal hypersurfaces to cut along.
	
	\begin{figure}[h!]
		\includegraphics[width=0.8\textwidth]{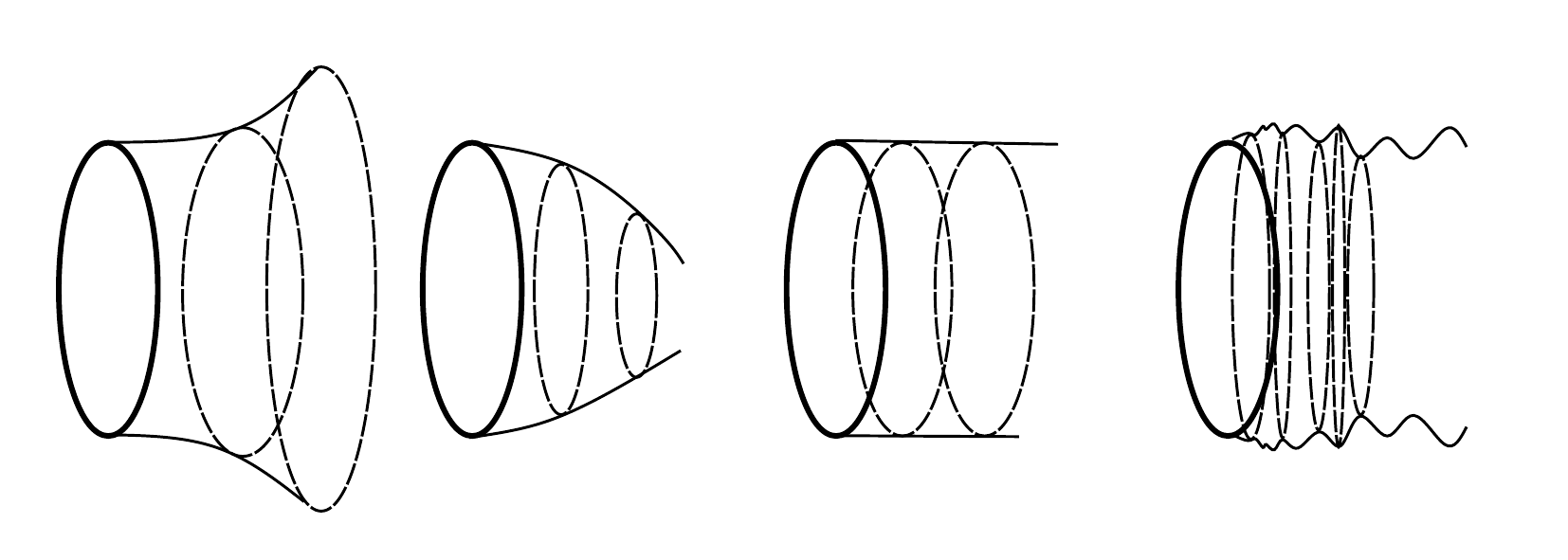}
		\caption{From the left to right: contracting neighborhood, expanding neighborhood, locally the neighborhood is foliated by embedded minimal hypersurfaces, and accumulating neighborhood.} \label{fig:nhbd}
	\end{figure}

	The implicit function theorem shows that a neighborhood of an embedded minimal hypersurface can be either contracting, expanding, locally foliated by minimal hypersurfaces, or accumulated by minimal hypersurfaces. The precise definitions of these concepts will be given in Definition \ref{def:4nbhds}. We refer the readers to Figure \ref{fig:nhbd} for the pictures of these neighborhoods. We will cut the manifold along contracting minimal hypersurfaces. There are several cases.

	\subsubsection*{Weakly Frankel case} The first possibility is that there is no minimal hypersurface which is contracting on a side. In this case, we say the manifold $(M,g)$ is \emph{weakly Frankel}. Recall that any two minimal hypersurfaces in a Frankel manifold intersect with each other. But in a weakly Frankel manifold, two minimal hypersurfaces can be disjoint, but these hypersurfaces must be connected by some minimal foliation.
	
	Furthermore, in weakly Frankel manifolds, we can show that for any $\omega>0$,
	\[\area_g(\cM^S_\omega):=\{\area_g(\Gamma):\text{$\Gamma$ is connected stable with $\area_g(\Gamma) \leq \omega$}\}\]
	is a finite set, see Corollary \ref{lem:WFStableAreas}. Therefore, either $\bigcup_{\omega>0}\area_g(\cM^S_\omega)$ is an infinite set, in which case we can find stable minimal hypersurfaces with arbitrarily large area, or $\bigcup_{\omega>0}\area_g(\cM^S_\omega)$ is a finite set, in which case there are only finitely many areas of stable minimal hypersurfaces. In the latter case, we need to find unstable minimal hypersurfaces with arbitrarily large area.
	
	Although the minimal hypersurface which realizes the $p$-width could have multiple components, but if so, these components are connected by a minimal foliation which implies they each are stable and have the same area. Thus,
	\[\omega_p=m_p\area_g(\Gamma_p)\] for some connected minimal hypersurface $\Gamma_p$ and multiplicity $m_p$. Moreover, Zhou observed that the unstable minimal hypersurfaces must have multiplicity-one, while the stable minimal hypersurfaces may possibly still have higher multiplicity (\cite[Theorem C]{Zh}). Thus, we only need to show that the $p$-width can not be all realized by multiples of stable minimal hypersurfaces for large enough $p$.
	
	To do this, we utilize Lusternik-Schnirelmann arguments shown in \cite{Ai}. In \cite{Ai} Aiex showed that the space of min-max minimal hypersurfaces in a manifold with positive Ricci curvature is noncompact. His idea is to study a generalization of the Lusternik-Schnirelmann category, and he can prove that under certain conditions, there is a strict jump of the width from $\omega_p < \omega_{p+N}$ for $p$ large. Then if there are eventually only stable minimal hypersurfaces appearing in the volume spectrum, we can show the growth violates the Weyl law. In this paper, we generalize Aiex ideas to weakly Frankel manifolds, without using the analyticity assumption in \cite{Ai}. One important new case which arises is when the manifold is foliated by minimal hypersurfaces. These can complicate the space of minimal hypersurfaces, and so we need a more precise study of its topology.

	\subsubsection*{Weak core case} The second possibility is that after cutting along finitely many contracting minimal hypersurfaces, we get a component that has minimal contracting boundary, and there are no contracting minimal hypersurfaces in the interior of this component. Motivated by Song's work, we call such a compact manifold a \emph{weak core}. If we can find a weak core, then we can still use Song's results with slight modifications in this more general setting to show that there exist embedded minimal hypersurfaces with arbitrarily large area.

	\subsubsection*{Accumulating case}
	
	The last possibility is that we can not get a weak core after finitely many cutting steps. Even when we cannot get a weak core, there are some relatively mild cases that we are able to handle. For example, there is a kind of metric called which we will call a \emph{spindle}, where the minimal hypersurface in the middle is non-isolated, and hypersurfaces have larger area the closer they are to the middle one (see Figure \ref{fig:spindle}). Although we may not be able to cut the manifold to a weak core in finitely many steps, we can use the approximation argument to show that there exist minimal hypersurfaces with arbitrarily large area.
	
	Finally, there is an extremely pathological case, that we can not get a weak core after finitely many cuttings, and we also can not find a spindle part. In this case, we can show that the space of certain pathological stable minimal hypersurfaces (see Figure \ref{fig:pathological}) is homeomorphic to a Cantor set. 

	\begin{figure}[h] \label{fig:spindle}
	\includegraphics[width=0.7\textwidth]{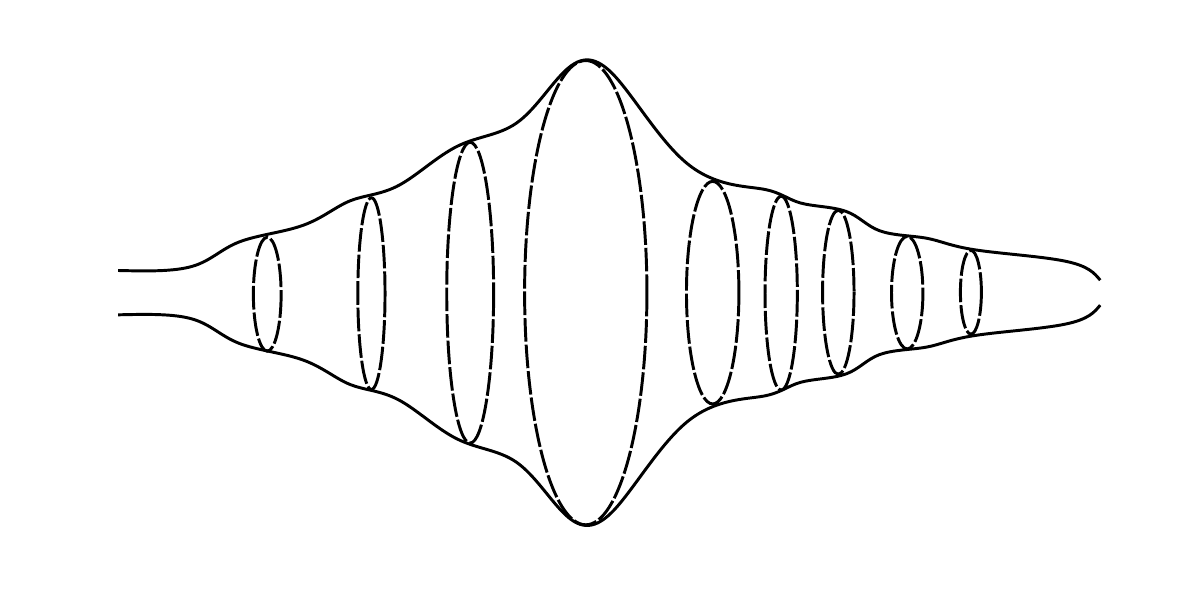}
	\caption{This is a model of a spindle metric on a cylinder. All the dashed lines represent stable minimal hypersurfaces. There are infinitely many minimal hypersurfaces approaching the center one.}
\end{figure}

	\begin{figure}[h] \label{fig:pathological}
		\includegraphics[width=0.9\textwidth]{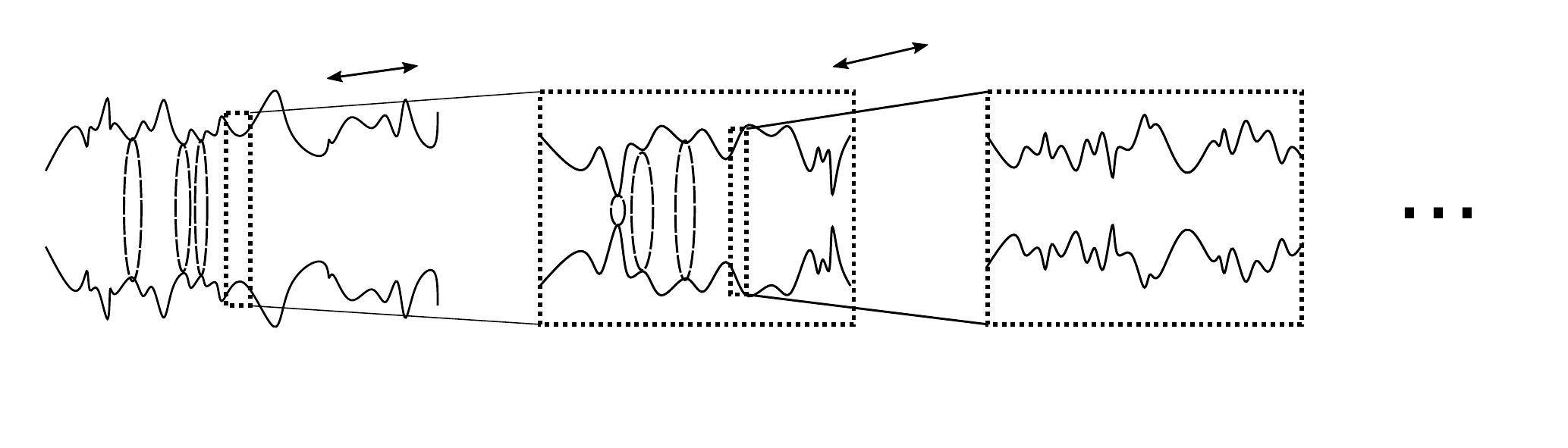}
		\caption{This diagram gives an idea of the fractal-like behavior of the pathological non-monotonic accumulating minimal hypersurfaces which appear in this the later case.}
	\end{figure}
	
	We give an explicit example of such a metric in the Appendix. Nevertheless, we cannot actually rule out the possibility that there are minimal hypersurfaces with arbitrarily large area in this case.

	\subsubsection*{One-sided minimal hypersurfaces}
	One-sided minimal hypersurfaces do not show up from min-max in a bumpy metric. This fact is proved by \cite{Zh}. In a general metric, one-sided minimal hypersurfaces may show up, but they are not trouble to the proof in the paper. But to make the presentation more clear, we will assume in Sections 2 through 5 that $(M,g)$ contains no one-sided minimal hypersurfaces. We discuss the necessary modifications to handle one-sided hypersurfaces in Section \ref{one-sided case}.

	\subsection*{Organization of the paper}
	
	In Section \ref{S:Preliminaries}, we give preliminary tools and definitions for this paper. In Section \ref{WeaklyFrankelCase}, we discuss the weakly Frankel case. In Section \ref{CoreCase}, we discuss the weak core case. In Section \ref{AccumlatingCase}, we discuss the pathological case. In Section \ref{one-sided case}, we discuss the case that one-sided minimal hypersurfaces show up. Finally, in Section \ref{S:Applications to special manifolds}, we discuss applications. We also have an appendix to record an explicit construction of the pathological metric.

\section{Preliminaries}\label{S:Preliminaries}

In Sections 2 through 5, we will assume that $(M,g)$ contains no one-sided minimal hypersurfaces. This is just to make the exposition easier to follow. But in Section 6, we describe the necessary technical modifications needed in general.

\subsection{Minimal hypersurfaces}

Let $\Gamma$ be a smooth two-sided closed embedded hypersurface of $M$. Recall that $\Gamma$ is a minimal hypersurface of $(M,g)$ if it is a critical point for the area functional, that is, for all smooth variations $\Gamma_t = \varphi(\Gamma \times \{t\})$ given by $\varphi: \Gamma \times (-\delta,\delta) \to M$ with $\varphi(x,0) = x$ for all $x \in \Gamma$, we have
\[
\left.\frac{d}{dt}\right|_{t=0} \area_{g}(\Gamma_t) = 0.
\]
Moreover, by the first variation formula,
\[
\left.\frac{d}{dt}\right|_{t=0} \area_{g}(\Gamma_t) = - \int_\Gamma \langle X,\bH \rangle
\]
where $X(x) = \frac{\partial \phi}{\partial t}(x,0)$ is the vector field on $\Gamma$ associated to the variation and $\bH$ is the mean curvature vector for $\Gamma$. So this says $\Gamma$ is a minimal hypersurface for $(M,g)$ if and only if its mean curvature vector $\bH$ vanishes identically.

Recall that a minimal hypersurface $\Gamma$ is stable if
\[
\left.\frac{d^2}{dt^2}\right|_{t=0} \area_{g}(\Gamma_t) \geq 0
\]
for all smooth variations of $\Gamma$. Let $\nu$ be a choice of unit normal on $\Gamma$. When $\Gamma_t$ is a normal variation of $\Gamma$ which means that we can write $X = f\nu$ for some $f \in C^\infty(\Gamma)$, then the second variation formula says
\[
\left.\frac{d^2}{dt^2}\right|_{t=0} \area_{g}(\Gamma_t) = \int_\Gamma |\nabla f|^2-|A|^2 f^2 - \Ric_M(\nu,\nu) f^2 = \int_\Gamma f L_\Gamma(f)
\]
where $A$ is the second fundamental form of $\Gamma$ and $L_\Gamma$ is the Jacobi operator
\[
L_\Gamma: C^\infty(\Gamma) \to C^\infty(\Gamma) \quad \text{given by} \quad L_\Gamma(f) = -\Delta f - (|A|^2+\Ric_M(\nu,\nu))f.
\]
Recall that the spectrum of $L_\Gamma$ is discrete, bounded from below, and the finite-dimensional eigenspaces give an orthogonal decomposition of $L^2(\Gamma)$. We let $\ind_g(\Gamma)$ denote the Morse index of a minimal hypersurface $\Gamma$ which is defined to be the number of negative eigenvalues of $L_\Gamma$. We say that a minimal hypersurface $\Gamma$ is degenerate if $0$ is an eigenvalue of $L_\Gamma$, that is, there exists a non-trivial $f$ such that $L_\Gamma(f)=0$ which we call a Jacobi field.

Recall that the eigenspace associated to the least eigenvalue is $1$-dimensional and the corresponding eigenfunction does not change sign, and hence, can be chosen to be positive. In particular, if we can find a non-trivial Jacobi field which changes sign, then $0$ is not the least eigenvalue of $L_\Gamma$ which implies that $\Gamma$ is unstable.

\begin{definition}\label{def:4nbhds}
	Let $\Gamma$ be a minimal hypersurface in $(M,g)$, and suppose that $\varphi: \Gamma \times [-\delta,\delta] \to M$ is a diffeomorphism onto its image such that $\varphi(x,0) = x$ for all $x \in \Gamma$. We say that the closed halved tubular neighborhood $\varphi(\Gamma \times [0,\delta])$ is
	\begin{itemize}
		\item \emph{contracting}: if the mean curvature of $\varphi(\Gamma \times \{t\})$ is non-vanishing and points towards $\Gamma$ for all $t \in (0,\delta]$;
		\item \emph{expanding}: if the mean curvature of $\varphi(\Gamma \times \{t\})$ is non-vanishing and points away from $\Gamma$ for all $t \in (0,\delta]$;
		\item \emph{foliated by minimal hypersurfaces}: if $\varphi(\Gamma \times \{t\})$ is a minimal hypersurface for all $t \in (0,\delta]$; and 
		\item \emph{accumulating}: if for each $t \in (0,\delta]$, the mean curvature of $\varphi(\Gamma \times \{t\})$ is either zero or non-vanishing, and moreover, exist arbitrarily small $r,s \in (0,\delta]$ such that $\varphi(\Gamma \times \{r\})$ is minimal and the mean curvature of $\varphi(\Gamma \times \{s\})$ is non-vanishing.
	\end{itemize}
	Similarly, we can define each of the above for the other closed halved tubular neighborhood $\phi(\Gamma \times [-\delta,0])$ by replacing each $(0,\delta]$ above with $[-\delta,0)$.
\end{definition}

Again, see Figure \ref{fig:nhbd} for models of these neighborhoods. It was noted in \cite[Lemma 11]{So}\footnote{In Song's case, he only needed to define and consider expanding and contracting neighborhoods because he assumed his manifold only had finitely many stable minimal hypersurfaces.} that each side of a minimal hypersurface must be given by one the four neighborhoods above and is generated by flowing along the first eigenfunction of $L_\Gamma$. For the reader's convenience, we give a proof of this in Appendix \ref{sec:NiceNeighborhood}. 

\begin{lemma}[\cite{So}] \label{TubularNeighborhood}
	Let $\Gamma \subset (M,g)$ be a minimal hypersurface. There exists a diffeomorphism $\varphi: \Gamma \times [-\delta,\delta] \to M$ onto a tubular neighborhood of $\Gamma$ such that
	\begin{itemize}
		\item $\varphi(x,0) = x$ for all $x \in \Gamma$;
		\item the closed halved tubular neighborhood $\varphi(\Gamma \times [0,\delta])$ is either contracting, expanding, foliated by minimal hypersurfaces, or accumulating; and
		\item the closed halved tubular neighborhood $\varphi(\Gamma \times [-\delta,0))$ is either contracting, expanding, foliated by minimal hypersurfaces, or accumulating.
	\end{itemize}
\end{lemma}

As expected, unstable minimal hypersurfaces will always be expanding on both sides, while strictly stable minimal hypersurfaces will always be contracting on both sides. For degenerate stable minimal hypersurfaces, any of the above can happen independently on either side.

The following lemma will help us find contracting minimal hypersurfaces.

\begin{lemma} \label{lem:MinimizeAreaHomology}
	Let $(N^{n+1},g)$ be a compact manifold with minimal boundary $\partial N$. If $\Sigma \subset \partial N$ is a non-contracting boundary component such that
	\[
	\area_g(\Sigma) \leq \area_g(\partial N \setminus \Sigma),
	\]
	then exactly one of the following must be true:
	\begin{enumerate}
		\item there exists minimal $\Gamma \subset \Int(N)$ which is contracting on a side, or
		\item the manifold $(N,g)$ is foliated by minimal hypersurfaces of the form $\Sigma_t = \varphi(\Sigma \times \{t\})$ for some diffeomorphism $\varphi: \Sigma \times [0,1] \to N$ such that $\Sigma_0 = \Sigma$.
	\end{enumerate}
\end{lemma}
\begin{proof}
	Let $\Sigma \subset \partial N$ be a non-contracting minimal hypersurface with $\area_g(\Sigma) \leq \area_g(\partial N \setminus \Sigma)$. Note that if $\Sigma$ is accumulating, then there must be a minimal hypersurface that is contracting on a side nearby (See Lemma \ref{lem:accumulating-nearby}) so that we are in the first case.
	
	Thus, now assume $\Sigma$ is either expanding or locally foliated by minimal hypersurfaces. In either case, by compactness \cite{SS} and Lemma \ref{TubularNeighborhood}, there exists some $\varepsilon \geq 0$ along with a diffeomorphism $\varphi:\Sigma \times [0,\varepsilon] \to N$ onto its image such that $\varphi(\Sigma\times\{0\}) = \Sigma$, the hypersurfaces $\varphi(\Sigma\times\{t\})$ are minimal for all $t \in [0,\varepsilon]$, and which cannot be extended as a minimal foliation any further. By the maximum principle, either $\Sigma_\varepsilon = \varphi(\Sigma\times\{\varepsilon\})$ is contained in $\Int(N)$ or is contained in $\partial N$. If $\Sigma_\varepsilon \subset \partial N$, then either $\varepsilon = 0$ or we must be in the second case where $N$ is foliated by minimal hypersurfaces.
	
	Now assume that either $\Sigma_\varepsilon \subset \Int(N)$ or $\Sigma_\varepsilon = \Sigma$. By construction, $\Sigma_\varepsilon$ is not foliated by minimal hypersurfaces outside of $\varphi(\Sigma \times [0,\varepsilon])$. Moreover, if $\Sigma_\varepsilon$ is contracting or accumulating on this side then we are in the first case. So finally assume that $\Sigma_\varepsilon$ is expanding on this side. Thus, if we minimize area in the homology class $[\Sigma] = [\Sigma_\varepsilon] \in H_n(N;\Z)$, we find a stable minimal hypersurface $\Gamma$ (with possibly multiple components) in $N$ such that
	\[
	\area_g(\Gamma) < \area_g(\Sigma_\varepsilon) = \area_g(\Sigma)
	\]
	since $\Sigma_\varepsilon$ is expanding on a side. By the maximum principle, each component of $\Gamma$ is either contained in $\Int(N)$ or $\partial N \setminus \Sigma$. There must be at least one connected component $\Gamma'$ of $\Gamma$ which is contained in $\Int(N)$. Otherwise if $\Gamma \subseteq \partial N \setminus \Sigma$, then $\Gamma = \partial N \setminus \Sigma$ since $[\Gamma] = [\Sigma] = [\partial N \setminus \Sigma] \in H_n(N;\Z)$ which gives a contradiction to
	\[
		\area_g(\Gamma) < \area_g(\Sigma) \leq \area_g(\partial N \setminus \Sigma).
	\]
	
	Thus, let $\Gamma'$ be a connected component of $\Gamma$ contained in $\Int(N)$. Like above, there exists some $\delta \geq 0$ along with a diffeomorphism $\varphi':\Gamma' \times [0,\delta] \to N$ onto its image such that $\varphi'(\Gamma'\times\{0\}) = \Gamma'$, the hypersurfaces $\varphi'(\Gamma'\times\{t\})$ are minimal for all $t \in [0,\delta]$, and which cannot be extended as a minimal foliation any further. By construction, $\Gamma'_\delta$ must be locally area minimizing outside of $\varphi'(\Gamma'\times [0,\delta])$. Thus, $\Gamma'_\delta$ is either contracting or accumulating on this side putting us in the first case.
\end{proof}

\subsection{Min-max theory}

Here is some notation we will use throughout. See \cite{MN1, MN3} for more details. Let $(M,g)$ be a closed Riemannian $(n+1)$-manifold, assumed to be isometrically embedded in $\R^L$. Let $\bI_k(M;\Z_2)$ denote the space of $k$-dimensional flat chains with $\Z_2$-coefficients in $\R^L$ with support in $M$. Consider
\[
\cZ_n(M;\Z_2) = \{T \in \bI_{n}(M;\Z_2) : T = \partial A \text{ for some } A \in \bI_{n+1}(M;\Z_2)\}
\]
the space of cycles\footnote{Technically, this is the space of boundaries rather than cycles. But this space of boundaries is the connected component of $0$ in the space of cycles which we will always work in.} which gives us a weak notion of hypersurfaces in $M$.

Let $\cV_n(M)$ be the closure, in the weak topology, of the space of rectifiable $n$-varifolds in $\R^L$ with support in $M$. The weak topology on varifolds is induced by the $\bF$ metric. For any $T \in \cZ_n(M;\Z_2)$, we denote $|T|$ to be the integral varifold associated to $T$, and use $\|T\|$ to denote the associated Radon measure. Likewise, for any $V \in \cV_n(M)$, we use $\|V\|$ to denote the associated Radon measure. 

We will consider three different topologies on the set $\cZ_n(M;\Z_2)$ induced by
\begin{itemize}
	\item the flat norm $\cF$,
	\item the $\bF$-metric given by $\bF(T,S) := \bF(|T|,|S|)+\cF(T-S)$, and
	\item the mass norm $\bM$.
\end{itemize}
Unless otherwise stated, we will consider $\cZ_n(M;\Z_2)$ with the flat topology.

In \cite{Al1}, Almgren proved the following are isomorphic
\[
\pi_k(\cZ_n(M;\Z_2)) \cong H_{n+k}(M;\Z_2) \cong 
\begin{cases} 
\Z_2 & k = 1, \\
0 & k \geq 2.
\end{cases}
\]
In fact, $\cZ_n(M;\Z_2)$ is weakly homotopy equivalent to $\RP^\infty$. Hence $H^*(\cZ_n(M;\Z_2);\Z_2)$ is given by the polynomial ring $\Z_2[\lambda]$ generated by the nonzero $\lambda \in H^1(\cZ_n(M;\Z_2);\Z_2)$. 

Let $X$ be a $k$-dimensional cubical subcomplex of $I^m = [0,1]^m$, and suppose $\Phi: X \to \cZ_n(M;\Z_2)$ is a flat continuous map. For $p \in \N$, we say that the map $\Phi$ is a \emph{$p$-sweepout} if 
\[
\Phi^*(\lambda^p) \neq 0 \in H^p(X;\Z_2)
\]
where $\lambda^p$ denotes the $p$-th cup product of $\lambda$. Moreover, we say such $\Phi: X \to \cZ_n(M;\Z_2)$ has no concentration of mass when
\[
\lim_{r \to 0} \sup \{\bM(\Phi(x) \cap B_r(p)) : x \in X, \ p \in M\} = 0.
\]
Given $p\in \N$, we use $\cP_p(M)$ to denote the set of $p$-sweepouts in $M$ with no concentration of mass, and define the $p$-width of $(M,g)$ as
\[
\omega_p(M,g)=\inf_{\Phi \in \cP_p(M)} \sup\{ \bM(\Phi(x)) : x \in \dmn(\Phi) \}.
\]
By much of the interpolation results proved by Marques-Neves (see \cite{MN3}), we can restrict the above to $\widetilde{\cP}_p(M) \subset \cP_p(M)$ of $p$-sweepouts which are continuous in the more strict $\bF$-topology.

Note that $\omega_{p}(M,g) \leq \omega_{p+1}(M,g)$ for all $p$ because any $(p+1)$-sweepout must be a $p$-sweepout. Moreover, these widths go to infinity. In fact, we can quantify the growth of these widths by the Weyl law of volume spectrum, which plays an important role in the study of the min-max theory of minimal hypersurfaces. 

\begin{theorem}[Gromov, Guth, and Liokumovich-Marques-Neves \cite{LMN}]
	There exist a constant $a(n) > 0$ such that for every closed $(n+1)$-manifold $(M,g)$,
	\[\lim_{p\to\infty}\omega_p(M,g)p^{-\frac{1}{n+1}}=a(n)\vol_g(M)^{\frac{n}{n+1}}.\]
\end{theorem}

Given a sequence of flat continuous maps $\{\Phi_i: X_i \to \cZ_n(M;\Z_2)\}$ where $X_i \subseteq I^{m_i}$ are cubical subcomplexes, we define the width of this sequence as
\[
\bL(\{\Phi_i\})=\limsup_{i\to\infty}\sup\{\bM(\Phi_i(x)):x\in X_i\}.
\]
and we define the set of critical varifolds
\[
\bC(\{\Phi_i\}):=
\{V \in \cV_n(M) : V = \lim_{j\to\infty} |\Phi_{i_j}(x_j)| \text{ and }  
\|V\|(M)=\bL(\{\Phi_i\})\}.
\]

We are also interested in min-max width of a fixed homotopy class. Given a fixed cubical subcomplex $X \subseteq I^{m}$ and $\bF$-continuous map $\Phi: X \to \cZ_n(M;\bF;\Z_2)$, we define the homotopy class of $\Phi$ as
\[
	\Pi(\Phi) = \{\Psi: X \to \cZ_n(M;\bF;\Z_2): \text{$\Psi$ is flat homotopic to $\Phi$}\}.
\] 
The width of a homotopy class $\Pi = \Pi(\Phi)$ is defined as
\[
	\bL(\Pi)=\inf_{\Psi \in \Pi} \sup_{x\in X} \bM(\Psi(x)).
\]

The following min-max theorem is the culmination of the work of Almgren-Pitts \cite{Pi}, the work of Marques-Neves \cite{MN3}, the multiplicity-one conjecture proven by Zhou \cite{Zh}, the denseness of bumpy metrics \cite{Wh}, along with the compactness theory of Sharp \cite{Sh}, and other important work.
\begin{theorem}[\cite{Zh} Theorem C] Let $(M^{n+1},g)$ have an arbitrary metric with $3 \leq n+1 \leq 7$, and let $\{\Phi_i\} \subset \Pi$ be a sequence of homotopic $p$-sweepouts such that
	\[
	\bL(\{\Phi_i\}) = \bL(\Pi) > 0.
	\]
	Then there exists a disjoint collection of closed embedded minimal hypersurfaces $\Sigma_1, \dots, \Sigma_N$ along with positive integer multiplicities $m_1,\dots,m_N$ such that
	\[
		\bL(\Pi) = \sum_{i=1}^N m_i\area_{g}(\Sigma_i) \quad \text{and} \quad \sum_{i=1}^{N} m_i\ind_g(\Sigma_i) \leq p
	\]
	where $m_i > 1$ only if $\Sigma_i$ is degenerate stable, and where as varifolds
	\[
	m_1 \Sigma_1 + \cdots + m_N \Sigma_N \in \bC(\{\Phi_i\}).
	\]
\end{theorem}

Although the theorem stated above is more general than \cite[Theorem C]{Zh}, it follows from the same reasoning but using \cite[Remark 5.9]{Zh} and \cite[Section 8]{MN3}.

\section{Weakly Frankel manifolds} \label{WeaklyFrankelCase}

In this section, we will study manifolds with the following property.

\begin{definition}
	We say a closed Riemannian manifold $(M,g)$ is \emph{weakly Frankel} there exist no minimal hypersurfaces which are contracting on a side.
\end{definition}

The goal of this section is to find minimal hypersurfaces with arbitrarily large area in a weakly Frankel manifold, see Theorem \ref{thm:weaklyFlarge}. Weakly Frankel is a generalization of the Frankel property. Recall that a manifold is Frankel if any two minimal hypersurfaces intersect with each other. In \cite{MN1} Marques-Neves proved that in a manifold with Frankel property, there exists infinitely many closed embedded minimal hypersurfaces.

The idea is that the Frankel property forces the min-max widths $\omega_p$ to be realized as $m_p\Sigma_p$ for some integer multiple $m_p$ of some \emph{connected} minimal hypersurface $\Sigma_p$. Then by using the growth of the widths along with some Lusternik-Schnirelmann arguments, they showed that min-max theory must give infinitely many distinct minimal hypersurfaces. We will use similar reasoning.

\subsection{Manifolds without contracting minimal hypersurfaces} \label{sec:NoContracting}

First, more generally, consider compact manifolds $(N^{n+1},g)$ with $3 \leq n+1 \leq 7$ such that:

\begin{itemize}
	\item the (possibly empty) boundary $\partial N$ is minimal and contracting, and
	\item $\Int(N)$ contains no minimal hypersurfaces which are contracting on a side.
\end{itemize}

\begin{remark}
	The second condition also implies that $N$ contains no minimal hypersurfaces which are accumulating on a side because such a hypersurface must be a limit of contracting minimal hypersurfaces, see Lemma \ref{lem:accumulating-nearby}.
\end{remark}

In particular, these following lemmas apply for a weakly Frankel manifold $N$ in the case when $N$ has no boundary. The case where $N$ has nonempty contracting minimal boundary $\partial N$ will be relevant in the later sections.

\begin{definition} \label{def:MinimalFoliation}
	We say connected closed hypersurfaces $\Sigma_0, \Sigma_1 \subset N$ are \emph{connected by a minimal foliation} if there exists $\delta \geq 0$ and $\varphi: \Sigma \times [0,\delta] \to N$ where
	\begin{itemize}
		\item $\varphi(\Sigma \times \{0\}) = \Sigma_0$ and $\varphi(\Sigma \times \{\delta\}) = \Sigma_1$,
		\item $\varphi$ is a diffeomorphism onto its image, and
		\item $\varphi(\Sigma \times \{t\})$ is a minimal hypersurface for all $t \in [0,\delta]$.
	\end{itemize}
\end{definition}

Unlike Frankel manifolds, weakly Frankel manifolds can have disjoint minimal hypersurfaces, but these hypersurfaces must be connected by a minimal foliation.

\begin{lemma} \label{lem:WFdisjoint}
	Assume $(N,g)$ as above. If $\Sigma_0, \Sigma_1 \subset \Int(N)$ are disjoint connected minimal hypersurfaces, then $\Sigma_0, \Sigma_1$ are connected by a minimal foliation.
\end{lemma}

\begin{proof}
	Assume that $\Sigma_0, \Sigma_1$ are disjoint minimal hypersurfaces. Consider the metric completion of $N \setminus (\Sigma_0 \cup \Sigma_1)$ and pick a connected component $W$ which contains two non-contracting minimal boundary components $\Gamma_0, \Gamma_1$ (along with possibly other necessarily contracting minimal boundary components coming from $N$) which are isometric to $\Sigma_0, \Sigma_1$ respectively. Without loss of generality, we may assume $\area_g(\Gamma_{0}) \leq \area_g(\Gamma_{1})$ so that 
	\[
	\area_g(\Gamma_{0}) \leq \area_g(\partial W \setminus \Gamma_{0})
	\]
	because $\Gamma_1 \subseteq \partial W \setminus \Gamma_0$. Finally, note since we are assuming that $\Int(N)$ contains no contracting minimal hypersurfaces, then so does $\Int(W)$. And therefore, $W$ must be foliated by minimal hypersurfaces by Lemma \ref{lem:MinimizeAreaHomology}. This gives the desired minimal foliation which connects $\Sigma_0, \Sigma_1$.
\end{proof}

\begin{remark} \label{rmk:WFSameArea}
	In particular, disjoint minimal hypersurfaces $\Sigma_0, \Sigma_1$ must be both degenerate stable, homologous to each other, and have $\area_g(\Sigma_0) = \area_g(\Sigma_1)$.
\end{remark}

Consider the set $\cM^S$ of all connected stable minimal hypersurfaces in $(\Int(N),g)$. For $\omega> 0$, it will be useful for us to define an equivalence relation on
\[
\cM^S_\omega = \{\Sigma \in \cM^S : \area_g(\Sigma) \leq \omega\}
\]
by $\Sigma_0 \sim \Sigma_1$ if and only if $\Sigma_0, \Sigma_1$ are connected by a minimal foliation. In particular, this means that if $\Sigma_0 \not\sim \Sigma_1$, then $\Sigma_0$ and $\Sigma_1$ must intersect by Lemma \ref{lem:WFdisjoint}.

\begin{lemma} \label{lem:WFFoliationClasses}
	Assume $(N,g)$ as above. For each $\omega > 0$, the set $\cM^S_\omega\!/\!\sim$ is finite.
\end{lemma}
\begin{proof}
	Suppose that $\cM^S_\omega\!/\!\sim$ were infinite. Then we could find a sequence $\Sigma_k \in \cM^S_\omega$ of stable minimal hypersurfaces each representing a different equivalence class. The areas being bounded implies that after relabeling, we can find a subsequence $\Sigma_k$ which converges smoothly to some stable $\Sigma \in \cM^S_\omega$ by \cite{SS}. Moreover, since $\Sigma$ is not connected by a minimal foliation to all but possibly one $\Sigma_k$, we can assume each $\Sigma_k$ intersects $\Sigma$ by Lemma \ref{lem:WFdisjoint}. But by \cite{Sh}, this means we can construct a Jacobi field for $\Sigma$ which changes sign which would contradict that $\Sigma$ is stable.
\end{proof}

\begin{corollary} \label{lem:WFStableAreas}
	Assume $(N,g)$ as above. For each $\omega> 0$, the set of stable areas
	\[
	 \area_g(\cM^S_\omega) = \{\area_g(\Sigma) : \Sigma \in \cM^S_\omega\}
	\]
	is finite. In particular, if the area of stable minimal hypersurfaces in $(N,g)$ is uniformly bounded, then the set of all stable areas $\area_g(\cM^S)$ is finite.
\end{corollary}
\begin{proof}
	This follows Lemma \ref{lem:WFFoliationClasses} and Remark \ref{rmk:WFSameArea} which tells us that any two stable minimal hypersurfaces in the same equivalence class must have the same area.
\end{proof}

Let $[\Sigma] \in \cM^S\!/\!\sim$ denote the equivalence class of stable minimal hypersurfaces which are connected to $\Sigma$ by a minimal foliation. There are three types of classes:
\begin{enumerate}
	\item We say $\Sigma$ is \emph{isolated} if $[\Sigma] = \{\Sigma\}$.
	\item We say $\Sigma$ generates a \emph{partial minimal foliation} if there exists a diffeomorphism $\varphi: \Sigma \times I \to M$ onto its image with $\{\Sigma_t\}_{t \in I} = [\Sigma]$ for $\Sigma_t = \varphi(\Sigma \times \{t\})$.
	\item We say $\Sigma$ generates a \emph{(full) minimal foliation} if there exists a fiber bundle $\pi: M \to S^1$ such that $\{\Sigma_\theta\}_{\theta \in S^1} = [\Sigma]$ where $\Sigma_{\theta} = \pi^{-1}(\{\theta\})$.
\end{enumerate}

\subsection{Space of stable minimal cycles}\label{SS:Space of stable minimal cycles}
Consider the subspace $\cS_\omega \subseteq \cZ_n(M;\Z_2)$ of all $T \in \cZ_n(M;\Z_2)$ such that $T = 0$ or $\spt(T)$ is an embedded stable minimal hypersurface in $(M,g)$ with $\bM(T) \leq \omega$. Note $T$ does \emph{not} need to be connected.

Given $\Sigma \in \cM^S_\omega$ which generates either a full or partial minimal foliation of $(M,g)$, we define the space $\cF_\omega \subseteq \cS_\omega$ \emph{associated to the foliation} to be the set of all $T \in \cS_\omega$ such that each connected component of $T$ equals some leaf in this foliation.

\begin{lemma} \label{lem:FullFoliationSpace}
	If $\Sigma$ generates a (full) minimal foliation of $(M,g)$, then the subspace $\cF_\omega \subseteq \cS_\omega$ associated to the foliation is homeomorphic to $\RP^m$ where $m$ is the largest even number such that $m \leq \omega/\area_g(\Sigma)$.
\end{lemma}
\begin{proof}
	Suppose $\Sigma$ generates a minimal foliation of $(M,g)$ so that there exist a fiber bundle $\pi: M \to S^1$ where the fibers $\Sigma_\theta := \pi^{-1}(\{\theta\})$ parameterize the foliation. Note that each $\Sigma_\theta$ is non-separating and homologous to $\Sigma$, and so each $T \in \cF_{\omega}$ must have an even number of components. Let $m$ be the largest even number with $m\area_g(\Sigma) \leq \omega$. Then for each $T \in \cF_{\omega}$, there exists $\theta_1, \dots, \theta_m \in S^1$ such that
	\[
	T = \Sigma_{\theta_1} + \Sigma_{\theta_2} + \cdots + \Sigma_{\theta_m}
	\]
	where we are considering $\Sigma_{\theta_i}$ as cycles in $\cZ_n(M;\Z_2)$. This means that the flat continuous map $q: (S^1)^m \to \cF_{\omega}$ given by
	\[
	q(\theta_1,\dots,\theta_m) = \Sigma_{\theta_1} + \cdots + \Sigma_{\theta_m}
	\]
	is surjective. Note $\spt(T)$ will have fewer than $m$ components when $\theta_i = \theta_j$ for some $i \neq j$, and $T = 0$ whenever $\theta_1 = \theta_i$ for all $i$. By identifying the fibers of this map, we get a continuous bijection (and hence a homeomorphism by compactness)
	\[
	f: TP^m(S^1) \to \cF_{\omega} \qquad \text{by} \qquad f[(\theta_1,\dots,\theta_m)] = \Sigma_{\theta_1} + \cdots + \Sigma_{\theta_m}
	\]
	where $TP^m(S^1) = (S^1)^m/\sim$ is given the quotient topology by the relation
	\[
	(\theta_1,\dots,\theta_m) \sim (\theta_1',\dots,\theta_m') \quad \text{iff} \quad \Sigma_{\theta_1} + \cdots + \Sigma_{\theta_m} = \Sigma_{\theta_1'} + \cdots + \Sigma_{\theta_m'}.
	\]
	The space $TP^m(S^1)$ is known as the $m$-th truncated symmetric product of $S^1$ and is homeomorphic to $\RP^m$ by \cite[Theorem 2]{Mo}. Therefore, $\cF_\omega \cong \RP^m$.
\end{proof}

\begin{lemma} \label{lem:PartialFoliationSpace}
	Suppose $\Sigma$ generates a partial minimal foliation of $(M,g)$, then the subspace $\cK_\omega \subseteq \cS_\omega$ associated to the foliation has two exactly components $\cK_\omega^0, \cK_\omega^1$ where $\cK_\omega^0$ strongly deformation retracts to $0 \in \cK_\omega^0$ and $\cK_\omega^1$ strongly deformation retracts to $\Sigma_0 \in \cK_\omega^1$.
\end{lemma}
\begin{proof}
	Suppose $\Sigma$ generates a partial minimal foliation of $(M,g)$ so that there exists a diffeomorphism $\varphi: \Sigma \times [0,\delta] \to M$ onto its image where the slices $\Sigma_t := \varphi(\Sigma \times \{t\})$ parameterize the partial foliation.
	
	First, note that $\Sigma$ here must be separating. Otherwise, we can consider the metric completion $N$ of $M \setminus \Sigma$ which must be connected and have exactly two boundary components both isometric to $\Sigma$. Moreover, if $N$ is foliated by minimal hypersurfaces, then $\Sigma$ must generate a (full) foliation of $M$ by Lemma \ref{lem:MinimizeAreaHomology}, which contradicts that this foliation is only partial.
	
	Thus, each $T \in \cK_\omega$ can have either an even or odd number of components. And, in fact, $\cK_\omega$ has exactly two connected components $\cK_\omega^0$, $\cK_\omega^1$ consisting of all $T \in \cK_\omega$ with an even number and odd number, respectively, of components. Consider the flat continuous homotopy $H: \cK_\omega \times [0,1] \to \cK_\omega$ given by
	\[
	H(T,s) = \Sigma_{(1-s)t_1} + \cdots + \Sigma_{(1-s)t_m}
	\]
	where $T = \Sigma_{t_1} + \cdots + \Sigma_{t_m}$. This map gives a strong deformation retract of $\cK_\omega^0$ to the point $0 \in \cK_\omega^0$ and a strong deformation retract of $\cK_\omega^1$ to the point $\Sigma_0 \in \cK_\omega^1$.
\end{proof}

\begin{proposition} \label{prop:StableTopology}
	There exist $C' > 0$ such that for all $\omega > 0$, we have 
	\[
	H^m(\cS_\omega,\Z_2) = 0 \quad \text{for all } \quad m \geq C'\omega.
	\]
\end{proposition}
\begin{proof}
	Let $\Sigma_F^{(1)},\dots,\Sigma_F^{(n_F)}$ generate all distinct full minimal foliations, $\Sigma_P^{(1)},\dots,\Sigma_P^{(n_P)}$ generate all distinct partial minimal foliations, and let $\Sigma_I^{(1)},\dots,\Sigma_I^{(n_T)}$ be all the isolated stable minimal hypersurfaces of $(M,g)$ with area less than or equal to $\omega$. Note there are indeed only finitely many such hypersurfaces by Lemma \ref{lem:WFFoliationClasses}.
	
	Given $\Sigma_F^{(i)}$ which generates a full minimal foliation, let $\cF^{(i)}_\omega$ denote the associated space of cycles. Similarly, given $\Sigma_P^{(j)}$ which generates a partial minimal foliation, let $\cK^{(j)}_\omega$ denote the associated space of cycles, where $\cK^{(j),0}$ and $\cK^{(j),1}$ are defined as in Lemma \ref{lem:PartialFoliationSpace}. Then
	\[
	\cS_\omega = \left(\bigvee_{i=1}^{n_F} \cF^{(i)}_\omega \vee \bigvee_{j=1}^{n_P} \cK_\omega^{(j),0}\right) \sqcup \bigsqcup_{j=1}^{n_P} \cK_\omega^{(j),1} \sqcup \bigsqcup_{k=1}^{n_I} \{\Sigma^{(k)}_I\}
	\]
	where the wedge sums are all taken at $0 \in \cZ_n(M;\Z_2)$. Since each $\cK_\omega^{(j),0}$ strongly deformation retracts to $0$ and each $\cK_\omega^{(j),1}$ strongly deformation retracts to $\Sigma_0$ by Lemma \ref{lem:PartialFoliationSpace}, we can construct a deformation retraction of $\cS_\omega$ onto the subspace
	\[
	\bigvee_{i=1}^{n_F} \cF^{(i)}_\omega \sqcup \bigsqcup_{j=1}^{n_P} \{\Sigma^{(j)}_P\} \sqcup \bigsqcup_{k=1}^{n_I} \{\Sigma^{(k)}_I\} \subseteq \cS_\omega.
	\]
	Now recall that each $\cF^{(i)}_\omega$ is homeomorphic to $\RP^{m_i}$ where $m_i$ is the largest even number less than or equal to $\omega/\area_g(\Sigma_F^{(i)})$ by Lemma \ref{lem:FullFoliationSpace}. Thus, for all $m \geq 1$,
	\[
	H^m(\cS_\omega;\Z_2) = H^m(\bigvee_{i=1}^{n_F} \cF^{(i)}_\omega;\Z_2) = \prod_{i=1}^{n_F} H^m(\cF^{(i)}_\omega;\Z_2) = \prod_{i=1}^{n_F} H^m(\RP^{m_i};\Z_2).
	\]
	
	Finally, by the monotonicity formula, there exists $C'>0$ such that $1/C' < \area_g(\Sigma)$ for all minimal hypersurfaces in $(M,g)$. Note $C'$ is defined independent of $\omega$ and that $m_i < C'\omega$ for all $i = 1,\dots,n_F$. Therefore, $H^m(\cS_\omega;\Z_2) = 0$ for all $m \geq C'\omega$ because then $H^m(\RP^{m_i};\Z_2) = 0$ for each $i$.
\end{proof}

\subsection{Lusternik-Schnirelmann arguments}

Let $\Lambda_\omega$ denote the subspace of varifolds $V \in \cV_n(M)$ with $\|V\|(M) \leq \omega$ and where $\spt(V)$ is a stable minimal hypersurface. Again, $V \in \Lambda_\omega$ may have multiple connected components where each component comes with some positive integer multiplicity.

The following lemma follows from the results proven by Aiex in \cite{Ai} which extend the arguments used in \cite{MN1}. It says that a map which stays near stable minimal hypersurfaces cannot be a $m$-sweepout for $m$ large enough because we have control on the topology of the space of stable minimal hypersurfaces.
\begin{lemma} \label{lem:NotSweepout}
	For every $\omega > 0$, there exists $\varepsilon > 0$ with the following property:
	
	If $Y$ is a cubical subcomplex and $\Psi: Y \to \cZ_n(M;\bF;\Z_2)$ is a map such that 
	\[
	\bF(|\Psi(y)|,\Lambda_\omega) < \varepsilon \quad  \text{for all}  \quad y \in Y, 
	\]
	then $\Psi$ is not an $m$-sweepout for $m \geq C' \omega$ where $C'$ given by Proposition \ref{prop:StableTopology}.
\end{lemma}
\begin{proof}
	In the proof of Proposition \ref{prop:StableTopology}, the subspace $\cS_\omega \subset \cZ_n(M;\Z_2)$ is deformation retraction to $\bigvee_{i=1}^{n_F} \cF^{(i)}_\omega \sqcup \bigsqcup_{j=1}^{n_P} \{\Sigma^{(j)}_P\} \sqcup \bigsqcup_{k=1}^{n_I} \{\Sigma^{(k)}_I\}$ that is homeomoprhic to the union of Riemannian manifolds. Then any closed curves in $\cS_\omega$ is flat homotopic to a closed curve in $\bigvee_{i=1}^{n_F} \cF^{(i)}_\omega \sqcup \bigsqcup_{j=1}^{n_P} \{\Sigma^{(j)}_P\} \sqcup \bigsqcup_{k=1}^{n_I} \{\Sigma^{(k)}_I\}$, where we can still apply Lemma 2 and Proposition 2 from \cite{Ai}, and as a consequence,
	\[
	\Ncat_{\cZ_n}(\cS_\omega) \leq \cat(\cS_\omega) \leq \dim(\cS_\omega) \leq m
	\] 
	for $m \geq C'\omega$ by Lemma \ref{prop:StableTopology} where $\Ncat_{\cZ_n}$ is a form of a relative Lusternik-Schnirelmann Category as defined in Section 3 of \cite{Ai}, $\cat(\cS_\omega)$ denotes the ordinary Lusternik-Schnirelmann Category, and $\dim(\cS_\omega)$ denotes the cohomological dimension. Finally, since $\cS_\omega \subset \cZ_n(M;\Z_2)$ is closed with $\Ncat_{\cZ_n}(\cS_\omega) \leq m$, the desired result follows from Lemmas 3 and 4 in \cite{Ai}.
\end{proof}

\begin{proposition} \label{prop:AiexResult}
	There exists $C > 0$ such that for each $p \in \N$, either
	\begin{enumerate}
		\item there exists a connected unstable minimal hypersurface $\Sigma \subset (M,g)$ with 
		\[
		\area_g(\Sigma) \geq \omega_p(M,g),
		\]
		\item or for all $m \geq C p^{\frac{1}{n+1}}$, we have that
		\[
		\omega_{p-m}(M,g) < \omega_p(M,g).
		\]
	\end{enumerate}
\end{proposition}
\begin{proof}
	Let $C'$ be given in Proposition \ref{prop:StableTopology} and Lemma \ref{lem:NotSweepout}, and pick $C>0$ such that $C'\omega_{p}(M,g) \leq C p^{\frac{1}{n+1}}$. Now assume that the second case above does not happen for this choice of $C$, that is, there exists an $m \geq Cp^{\frac{1}{n+1}}$ such that 
	\[
	\omega_{p-m}(M,g) = \omega_p(M,g).
	\] We will consider two separate cases.
	
	{\bf Case 1:} Suppose there exists a homotopy class $\Pi$ of $p$-sweepouts with
	\[
	\bL(\Pi) = \omega_{p}(M,g).
	\] 
	Thus, there exist a cubical subcomplex $X$ along with a min-max sequence of flat homotopic $p$-sweepouts $\Phi_i: X \to \cZ_n(M;\bF;\Z_2)$ such that
	\[
	\bL(\{\Phi_i\}) = \bL(\Pi) = \omega_{p}(M,g).
	\]
	Moreover, we can assume that the sequence is \emph{pulled-tight} by \cite[Section 2.8]{MN3}, that is, so that every $V \in \bC(\{\Phi_i\})$ is stationary. Choose $\varepsilon > 0$ as in Lemma \ref{lem:NotSweepout} for $\omega = \omega_{p}(M,g)$. Let $Y_i$ be the cubical subcomplex consisting of all cells $\alpha$ in $X$ with 
	\[
	\bF(|\Phi_i(x)|, \Lambda_\omega) < \varepsilon \quad \text{ for all } x \in \alpha,
	\] and consider the cubical subcomplex $Z_i := \overline{X \setminus Y_i}$. Note that $\bF(|\Phi_i(x)|, \Lambda_\omega) \geq \varepsilon/2$ for all $x \in Z_i$. By the choice of $\varepsilon$ and $C$, the maps $\Phi_i|_{Y_i}$ are not $m$-sweepouts for all $m \geq Cp^{\frac{1}{n+1}}$ by Lemma \ref{lem:NotSweepout}. Therefore, by the Vanishing Lemma \cite[Lemma 1]{Ai}, the maps $\Phi_i|_{Z_i}$ must be $(p-m)$-sweepouts. In particular, $\bL(\{\Phi_i|_{Z_i}\}) \geq \omega_{p-m}(M,g)$. Moreover, we must have equality $\bL(\{\Phi_i|_{Z_i}\}) = \omega_{p-m}(M,g)$ because
	\[
	\bL(\{\Phi_i|_{Z_i}\}) \leq \bL(\{\Phi_i\}) = \omega_{p}(M,g) = \omega_{p-m}(M,g).
	\]
	
	Since the sequence $\Phi_i$ is pulled-tight, every $V \in \bC(\{\Phi_i|_{Z_i}\})$ is stationary as well. Furthermore, note if there were no $V \in \bC(\{\Phi_i|_{Z_i}\})$ such that $\spt(V)$ is a smooth embedded minimal hypersurface, then by the regularity theory of Pitts \cite{Pi}, no $V \in \bC(\{\Phi_i|_{Z_i}\})$ is almost minimizing in annuli. However, Pitts’ combinatorial argument\footnote{It is important that all $Z_i$ here are cubical subcomplexes of $X \subset I^m$ where $m$ is fixed.} (see \cite[Theroem 5]{Ai}) would allow us to find a sequence of $(p-m)$-sweepouts $\Psi_i: Z'_i \to \cZ_n(M;\bF;\Z_2)$ such that 
	\[
	\bL(\{\Psi_i\}) < \bL(\{\Phi_i|_{Z_i}\}) = \omega_{p-m}(M,g)
	\]
	which contradicts the definition of the $(p-m)$-width.
	
	Therefore, there must exist a $V \in \bC(\{\Phi_i|_{Z_i}\})$ such that
	\[
	\|V\|(M) = \bL(\{\Phi_i|_{Z_i}\}) = \omega_p(M,g)
	\]
	and where $\spt(V)$ is a smooth embedded minimal hypersurface. However, by construction, $V \not\in \Lambda_\omega$. Thus, $\Gamma := \spt(V)$ must be unstable, and hence, $\Gamma$ is connected by Lemma \ref{lem:WFdisjoint} and has
	$\area_g(\Gamma) = \omega_p(M,g)$.
	
	{\bf Case 2:} Now, suppose for every homotopy class $\Pi$ of $p$-sweepouts, we have
	\[
	\omega_p(M,g) < \bL(\Pi).
	\]
	In particular, we can find homotopy classes $\Pi_i$ of $p$-sweepouts such that
	\[
	\omega_p(M,g) < \cdots < \bL(\Pi_i) < \cdots < \bL(\Pi_2) < \bL(\Pi_1) \leq \omega_p(M,g)+1.
	\]
	Note each $\bL(\Pi_i) = m_i \area_g(\Gamma_i)$ for some minimal hypersurfaces $\Gamma_i$ where $m_i = 1$ if $\Gamma_i$ is unstable. Note that some $\Gamma_i$ must be unstable because there would be infinitely many distinct areas of stable minimal hypersurfaces of bounded area which violates Lemma \ref{lem:WFStableAreas}. So take $\Gamma = \Gamma_i$ unstable so that $m_1 = 1$. Therefore, we have $\area_g(\Gamma) = \bL(\Pi_i) > \omega_p(M,g)$.
\end{proof}

\begin{theorem}\label{thm:weaklyFlarge}
	If $(M,g)$ is weakly Frankel, then there exists connected minimal hypersurfaces of arbitrarily large area.
\end{theorem}
\begin{proof}
	Assume the area of unstable minimal hypersurfaces in $(M,g)$ is bounded. Take $C > 0$ from Proposition \ref{prop:AiexResult} and pick $p_0 \in \N$ such that $\omega_{p_0}(M,g)$ is strictly larger than the area of any unstable minimal hypersurface so that (1) does not happen in Proposition \ref{prop:AiexResult} for $p$ large. Then for all $p \geq p_0$, we have
	\[
	\omega_{p-m_0}(M,g) < \omega_{p}(M,g)
	\]
	where $m_0 = \lceil Cp^{\frac{1}{n+1}} \rceil$. Note that this implies that
	\[
	\omega_{p-\ell m_0}(M,g) < \cdots < \omega_{p-2m_0}(M,g) < \omega_{p-m_0}(M,g) < \omega_{p}(M,g)
	\]
	where $\ell = \lfloor(p-p_0)/m_0\rfloor$. Moreover, by the choice of $m_0$, there exists $q_0 \geq p_0$ and $0 < C' < C$ such that $\ell \geq C' p^\frac{n}{n+1}$ for all $p$. Therefore, for all $p \geq q_0$,
	\[
	\#\{\omega_{k}(M,g) : p_0 \leq k \leq p\} \geq \ell \geq C' p^\frac{n}{n+1}.
	\]
	
	Now, we will show that this forces the area of stable minimal hypersurfaces to be unbounded. Otherwise, then by Lemma \ref{lem:WFStableAreas}, there are only finitely many possible stable areas $\{\alpha_1,\dots,\alpha_N\}$. By the choice of $p_0$, we know that each width with $p \geq p_0$ is of the form $\omega_{p}(M,g) = m \area_g(\Sigma)$ for some $m \in \N$ and some stable minimal hypersurface $\Sigma$ by Lemma \ref{lem:WFdisjoint}. However, this implies
	\[
	\#\{\omega_{k}(M,g) : p_0 \leq k \leq p\} \leq \#\{m\alpha_i : m \in \N, 1 \leq i \leq N, m\alpha_i \leq \omega_{p}(M,g)\}.
	\]
	Let $\alpha = \min_i \alpha_i$ so that if $m\alpha_i \leq \omega_{p}(M,g)$, then $m \leq \frac{\omega_{p}(M,g)}{\alpha} \leq \frac{C}{\alpha}p^{\frac{1}{n+1}}$. Thus,
	\[
	\#\{m\alpha_i : m \in \N, 1 \leq i \leq N, m\alpha_i \leq \omega_{p}(M,g)\} \leq C''p^{\frac{1}{n+1}}
	\]
	where $C'' = \frac{CN}{\alpha}$. But this gives a contradiction.
\end{proof}

\section{Weak core manifolds} \label{CoreCase}

Song resolved Yau's conjecture by proving in \cite{So} that for any $(M^{n+1},g)$ with $3 \leq n+1 \leq 7$, there exists infinitely many distinct minimal hypersurfaces. Song did so by building upon the work of Marques-Neves \cite{MN1} in the Frankel case. 

His idea is that if $(M,g)$ not Frankel and only has finitely many stable minimal hypersurfaces, then we can cut $M$ along some contracting minimal hypersurfaces to construct a compact manifold $U$ with contracting boundary which satisfies the Frankel property for minimal hypersurfaces in $\Int(U)$. He then introduced a novel min-max theory which generates minimal hypersurfaces in $\Int(U)$ by essentially doing min-max on a non-compact manifold formed by gluing cylindrical ends to $U$.

Although the widths $\widetilde{\omega}_p$ in Song's min-max theory grow linearly (instead of sublinearly), by the nature of his construction, Song was able to say more delicate information about the growth of these widths which allows him to find infinitely many minimal hypersurfaces. 

In this section we adapt Song's ideas, but using our more general notion of weak cores. Many of Song's results carry over to this case without much more work.

\subsection{Song's min-max theory}
First, we will briefly overview Song's min-max theory introduced in \cite{So}. Let $(U,g)$ be compact with nonempty minimal contracting boundary. Consider the complete non-compact cylindrical extension
\[
\cC(U) = U \cup (\partial U \times [0,\infty))
\]
where we identify the boundaries $\partial U \subset U$ and $\partial U \times \{0\} \subset \partial U \times [0,\infty)$ to each other in the obvious way. Give $(\cC(U),h)$ the metric $g$ on $U$ and the product metric on $\partial U \times [0,\infty)$. Note that $\cC(U)$ may not be smooth where we glued at.

Let $K_1 \subset K_2 \subset \cdots$ be an exhaustion of $\cC(U)$ by compact regions with smooth boundary. We can consider the corresponding ordinary min-max widths\footnote{Although $K_i$ has boundary, we can still define the widths $\omega_p(K_i,h)$ and by the min-max theory of Li-Zhou \cite{LZ} (see also \cite{SWZ}), we can find minimal hypersurfaces realizing the widths. However, these may have boundary which touches the boundary of $K_i$.} to define
\[
\widetilde{\omega}_p(U,g) := \omega_p(\cC(U),h) := \lim_{i \to \infty} \omega_p(K_i,h)
\]
which Song shows is independent of the choice of exhaustion. Although we are considering sweepouts of the complete non-compact space $\cC(U)$, Song shows that we realize the widths by minimal hypersurfaces with support in $\Int(U) \subset \cC(U)$. Since the time Song gave his proof, the multiplicity-one conjecture has been settled in \cite{Zh} and \cite{SWZ} which gives the following result.

\begin{proposition}[\cite{So} Theorem 9, \cite{SWZ} Theorem 1.1] \label{prop:Song_minmax}
	Let $(U,g)$ be compact with nonempty minimal contracting boundary, and let $p \in \N$ be fixed. Then there exist a collection of disjoint connected minimal hypersurfaces $\Gamma_1,\dots,\Gamma_N \subset \Int(U)$ along with positive integer multiplicities $m_1,\dots,m_N$ such that
	\[
	\widetilde{\omega}_p(U,g) = \sum_{i=1}^N m_i \area_g(\Gamma_i) \quad \text{and} \quad \sum_{i=1}^N \ind_g(\Gamma_i) \leq p
	\]
	where $m_i = 1$ whenever the component $\Gamma_i$ is unstable.
\end{proposition}
\begin{proof}
	The proof is the same as \cite[Theorem 9]{So}, except at the time of Song's proof, the index bounds and multiplicity-one result in \cite[Theorem 1.1]{SWZ} were not available.
\end{proof}

\subsection{Weak core manifolds}

We need to generalize the notation of Song's cores.

\begin{definition}
	We say a compact manifold $(U,g)$ with nonempty minimal contracting boundary is a \emph{weak core} if there exist no minimal hypersurfaces contained in $\Int(U)$ which are contracting on a side.
\end{definition}

\begin{corollary} \label{cor:WC_minmax}
	Let $(U,g)$ be a weak core. For $p \in \N$, there exists a connected minimal hypersurface $\Gamma_p$ in the interior of $U$ and an positive integer $m_p$ such that
	\[
	\widetilde{\omega}_p(U,g) = m_p \area_g(\Gamma_p) \quad \text{and} \quad \ind_g(\Gamma_p) \leq p
	\]
	where $m_p = 1$ whenever $\Gamma_p$ is unstable.
\end{corollary}
\begin{proof}
	Fix $p \in \N$. By Proposition \ref{prop:Song_minmax}, there exist disjoint connected minimal hypersurfaces $\Gamma_1,\dots,\Gamma_N \subset \Int(U)$ with positive integers $m_1,\dots,m_N$ such that
	\[
	\widetilde{\omega}_p(U,g) = \sum_{i=1}^N m_i \area_g(\Gamma_i).
	\]
	Since $(U,g)$ is a weak core and $\Gamma_1,\dots,\Gamma_N$ are disjoint, we must have that
	\[
	\area_g(\Gamma_1) = \area_g(\Gamma_2) = \cdots = \area_g(\Gamma_N).
	\]
	by Remark \ref{rmk:WFSameArea}. Now, let $\Gamma = \Gamma_1$ and $m = m_1 + \cdots + m_N$ so that
	\[
	\widetilde{\omega}_p(U,g) = \sum_{i=1}^N m_i \area_g(\Gamma_i) = m \area_g(\Gamma).
	\]
	Finally, note if $m > 1$, then either $N > 1$ or $m_1 > 1$. In the first case, $\Gamma$ is stable by Remark \ref{rmk:WFSameArea}, while $\Gamma = \Gamma_1$ must be stable by Proposition \ref{prop:Song_minmax} in the second case.
\end{proof}

The largest area boundary component of $U$ plays an important role in Song's proof. This lemma is proved exactly like \cite[Lemma 12]{So}, but using Lemma \ref{lem:MinimizeAreaHomology}.
\begin{lemma}[\cite{So}] \label{lem:WCLargerThanBoundary}
	Let $(U,g)$ be a weak core, and let $A$ be the area of the largest connected component of $\partial U$. For any minimal hypersurface $\Gamma \subset \Int(U)$,
	\[
	\area_g(\Gamma) > A.
	\]
\end{lemma}

\begin{theorem}\label{thm:WC->large_area}
	If $(U,g)$ is a weak core, then there exist connected minimal hypersurfaces in $U$ of arbitrarily large area.
\end{theorem}
\begin{proof}
	We prove by contradiction. Assume that the area of stable minimal hypersurfaces in $(U,g)$ is uniformly bounded. Let $A$ denote the area of the largest component of $\partial U$. Then by \cite[Theorem 9]{So}, there exists some constant $B$ (which is independent of $p$) such that
	\[
	Ap \leq \widetilde{\omega}_p(U,g) \leq Ap+Bp^{\frac{1}{n+1}} \qquad \text{and} \qquad \widetilde{\omega}_{p+1}(U,g) \geq \widetilde{\omega}_p(U,g) + A
	\]
	for all $p$. By Corollary \ref{cor:WC_minmax}, we can find a connected minimal hypersurface $\Gamma_p$ contained the interior of $U$ along with a positive integer $m_p$ such that
	\[
	\widetilde{\omega}_p(U,g) = m_p \area_g(\Gamma_p)
	\]
	where $m_p = 1$ whenever $\Gamma_p$ is unstable. Moreover, by Lemma \ref{lem:WCLargerThanBoundary}, we have that $\area_g(\Gamma_p) > A$ for each $p$. Consider the set of areas (ignoring multiplicities) which appear in this sequence, that is, the set
	\[
	\{ \area_g(\Gamma_p) : p \in \N \}.
	\]
	Note this set must be infinite because if not, then by \cite[Lemma 14]{So}, there would exist an $\varepsilon > 0$ such that for all $p$ large enough,
	\[
	\widetilde{\omega}_p(U,g) > (A+\varepsilon)p
	\]
	which contradicts the growth of the widths. Recall there are only finitely many distinct areas of stable minimal hypersurfaces in $(U,g)$ by Lemma \ref{lem:WFStableAreas} and our assumed stable area bound. Therefore, we can find an increasing sequence $p_k$ such that each $\Gamma_{p_k}$ is unstable so that $m_{p_k} = 1$, and thus,
	\[
	Ap_k \leq \widetilde{\omega}_{p_k}(U,g) = \area_g(\Gamma_{p_k})
	\]
	for all $k$. And so, $\area_g(\Gamma_{p_k}) \to \infty$ as $k \to \infty$. This is a contradiction.
\end{proof}

\section{Accumulating case} \label{AccumlatingCase}

\subsection{Accumulating minimal hypersurfaces}

In this section, we will assume that $(M,g)$ is not weakly Frankel, does not contain a weak core, and has a uniform area bound for connected stable minimal hypersurfaces. Otherwise, we can find minimal hypersurfaces of arbitrarily large area by Sections \ref{WeaklyFrankelCase} and \ref{CoreCase}.

The following lemma is proven just like \cite[Lemma 12]{So}, but using Lemma \ref{lem:MinimizeAreaHomology}.
\begin{lemma}[\cite{So}] \label{lem:CutAlongContracting}
	Let $(N,g)$ be compact manifold with (possibly empty) contracting minimal boundary, and let $\Gamma \subset \Int(N)$ be a minimal hypersurface which is contracting on a side. We can cut $N$ along $\Gamma$ (and possibly finitely many other minimal hypersurfaces) to obtain a compact manifold $(N',g)$ with contracting minimal boundary with a component $\Sigma \subseteq \partial N'$ isometric to $\Gamma$.
\end{lemma}

\begin{definition}
	We say that $U$ is a \emph{Song region} of $(M,g)$ if $(U,g)$ is a compact manifold with nonempty contracting boundary formed by \emph{cutting} $M$ along some collection of disjoint (stable) minimal hypersurfaces $\Sigma_1, \dots, \Sigma_N$, that is, consider some connected component of the metric completion of 
	\[
	M \setminus (\Sigma_1 \cup \cdots \cup \Sigma_N).
	\]
	We say that a Song region $W$ is a \emph{sub-Song region} of a Song region $U$ if $W$ is formed by cutting $U$ along additional disjoint minimal hypersurfaces (contained in $\Int(U)$).
\end{definition}

Like in Section \ref{sec:NoContracting}, given a compact Riemannian manifold $(N,g)$, we define
\[
\cM^S(N) = \{\Sigma \subset \Int(N) : \text{$\Sigma$ connected stable minimal hypersurface in $(N,g)$}\}.
\]
Since we are assuming a uniform area bound for stable minimal hypersurfaces, this space is strongly compact in the smooth topology \cite{SS} (see also \cite{Sh}). We will be interested in various subspaces coming from different Song regions $U$ of $(M,g)$.

Observe our assumption that no Song region $U$ of $(M,g)$ is a weak core implies
\[
\cM^C(U) = \{\Sigma \in \cM^S(U) : \text{$\Sigma$ is contracting on a side}\}.
\]
is always nonempty. Note that we can consider $\cM^C(U)$ as a subspace of the compact space $\cM^S(M)$ of all connected stable minimal hypersurfaces in $(M,g)$. However, note that this subspace is not closed in our case. In fact, if we denote
\[
\cM^A(U) = \{\Sigma \in \cM^S(U) : \text{$\Sigma$ is accumulating on a side}\},
\]
then we have that $\cM^A(U)$ is precisely the limit points of $\cM^C(U)$, see Lemma \ref{lem:accumulating-nearby}. In fact, the limit of minimal surfaces contracting on a side can not be a minimal surfaces that is expanding or has neighborhood locally foliated by minimal surfaces, because we can express the neighborhood of these minimal surfaces by a foliation, and the foliation has either zero mean curvature or the mean curvature pointing away from the minimal surface. In both situation, maximum principle rules out the existence of a minimal surface that is contracting from a side. Thus the closure
\[
\overline{\cM}^C(U) = \cM^C(U) \cup \cM^A(U).
\]

By iteratively applying Lemma \ref{lem:CutAlongContracting}, we can see that $\cM^A(U)$ must always be nonempty as well. In fact, we will look at the largest area minimal hypersurface in $\cM^A(U)$, and either, it will be very pathological, or it will have a nice enough local structure will allow us to find large minimal hypersurfaces. 

\begin{lemma} \label{NoncontractingAccumulating}
	Let $U$ be a Song region of $(M,g)$ as above. Then there exists a stable minimal hypersurface $\Gamma \subset \Int(U)$ which is accumulating on a side and is not contracting on the other side. In other words, $\Gamma \in \cM^A(U) \setminus \cM^C(U)$.
\end{lemma}
\begin{proof}
	Since $M$ has no weak core, the set $\cM^C(U)$ is nonempty. By compactness of stable minimal hypersurfaces, we can find $\Gamma$ in the closure of $\cM^C(U)$ such that
	\[
	\area_g(\Gamma) = \sup\{\area_g(\Sigma) : \Sigma \in \cM^C\}.
	\]
	We will show that $\Gamma$ is the desired minimal hypersurface. Note since $\Gamma$ is in the closure of $\cM^C(U)$, it suffices to show that $\Gamma$ is not contracting on a side because then $\Gamma$ must be a limit point of $\cM^C(U)$ and hence also accumulating on a side.
	
	So suppose otherwise that $\Gamma$ is contracting on a side. Let $W_0$ be the connected component of the metric completion of $U \setminus \Gamma$ which has a contracting boundary component $\Sigma$ isometric to $\Gamma$. We will inductively cut $W_0$ along finitely many minimal hypersurfaces to find weak core $W = W_N$ which would give a contradiction. 
	
	Given $W_i$ containing this same boundary component $\Sigma$. By compactness, there is a stable minimal hypersurface $\Sigma_i \in \overline{\cM}^C(W_i)$ such that
	\[
	d_H(\Sigma,\Sigma_i) = \inf\{d_H(\Sigma,\Sigma') : \Sigma' \in \cM^C(W_i)\}
	\]
	where $d_H$ denotes Hausdorff distance in $W_{i}$. Note that this minimizer $\Sigma_i$ has a positive distance from $\Sigma$ by the maximum principle applied to the contracting neighborhood of $\Sigma$. In fact, suppose $\Sigma_i$ is the limit of $\Sigma'_j\in\cM^C(W_i)$. If $d_H(\Sigma,\Sigma_i) =0$, by maximum principle, $\Sigma_i=\Sigma$, which implies that $\Sigma_j'$ converge to $\Sigma$ smoothly. This implies that when $j$ is sufficiently large, $\Sigma_j'$ enters the contracting neighborhood of $\Sigma$, which is a contradiction. Now, we form $W_{i+1}$ by cutting $W_i$ along $\Sigma_i$ and picking the connected component containing $\Sigma$ as a boundary component.
	
	We will now see that the above process must eventually terminate. Otherwise, we obtain a sequence of disjoint stable $\Sigma_i \in \overline{\cM}^C(W_i) \subseteq \overline{\cM}^C(W_0)$. By compactness, there is a smoothly and graphically convergent subsequence. Therefore, we can find $N_0 < N_1 < N_2$ and closed cylindrical regions $R_0$ and $R_1$ such that
	\[
	\partial R_0 = \Sigma_{N_0} \cup \Sigma_{N_1}, \qquad \partial R_1 = \Sigma_{N_1} \cup \Sigma_{N_2}, \qquad R_0 \cap R_1 = \Sigma_{N_1}.
	\]
	Observe that $W_{N_1}$ contains a boundary component isometric to $\Sigma_{N_0}$ which has the cylindrical neighborhood $R_0 \cup R_1$ contained in $W_{N_1}$. Since the $R_0$ part touches a boundary component not equal to $\Sigma$, then $\Sigma_{N_2}$ must strictly closer than $\Sigma_{N_1}$ to $\Sigma$ inside $W_{N_1}$. However, this would contradict our choice of $\Sigma_{N_1}$. 
	
	Therefore, this inductive procedure must give a compact region $W = W_N$ such that $\cM^C(W)$ is empty. Finally, to show that $W$ is weak core---in order to reach a contradiction to our assumption that $\Gamma$ is contracting on a side---we just need to show that the other components of $\partial W$ are also contracting.
	
	So consider any other boundary component $\Sigma' \subseteq \partial W \setminus \Sigma$. Note that
	\[
		\area_g(\Sigma') \leq \area_g(\Gamma) = \area_g(\Sigma) \leq \area_g(W \setminus \Sigma')
	\]
	by the choice of $\Gamma$. By the construction of $W$, there are no contracting minimal hypersurfaces in $\Int(W)$ and $W$ is not foliated by minimal hypersurface (because $\Sigma \subseteq \partial W$ is contracting). Therefore, we reach a contradiction by Lemma \ref{lem:MinimizeAreaHomology}.
\end{proof}

\subsection{Spindle minimal hypersurfaces}

As mentioned, minimal hypersurfaces found by Lemma \ref{NoncontractingAccumulating} may have a nice enough structure which is modeled by Figure \ref{fig:spindle}.
\begin{definition} \label{def:MonotoneSaddle}
	We say a minimal hypersurface $\Gamma \subset (M,g)$ is a \emph{spindle} if there exists $0< \delta$ and a map $\varphi: \Gamma \times [-\delta,\delta] \to M$ such that
	\begin{itemize}
		\item $\varphi$ is a diffeomorphism onto its image,
		\item $\varphi(\Gamma \times \{t\})$ either is minimal or has non-vanishing mean curvature for all $t$,
		\item $\varphi(\Gamma \times \{0\}) = \Gamma$,
		\item $\varphi(\Gamma \times \{-\delta\})$ and $\varphi(\Gamma \times \{\delta\})$ have non-vanishing mean curvature pointing away from $\Gamma$,
		\item $\area_g(\varphi(\Gamma \times \{t\}))$ is weakly increasing for $t \in [-\delta,0]$, and
		\item $\area_g(\varphi(\Gamma \times \{t\}))$ is weakly decreasing for $t \in [0,\delta]$.
	\end{itemize}
\end{definition}

We will show that if $(M,g)$ has a spindle which is accumulating on a side, then we can find a Song region which looks like a weak core after applying certain small conformal perturbations.

\begin{proposition} \label{prop:spindle}
	If $(M,g)$ has a spindle minimal hypersurface $\Gamma$ which is accumulating on a side, then $M$ has minimal hypersurfaces of arbitrarily large area.
\end{proposition}
\begin{proof}
	If $\Gamma$ is a spindle, then there is a $\varphi: \Gamma \times [-\delta,\delta] \to M$ as in Definition \ref{def:MonotoneSaddle}. Let $\Gamma_t := \varphi(\Gamma \times \{t\})$. If $\Gamma$ is accumulating on a side, we can also assume that $\Gamma = \Gamma_0$ and that there exist arbitrarily small $s \in (-\delta,0)$ such that $\Gamma_s$ is minimal and contracting inside $\phi(\Gamma \times [s,0])$.
	
	Following \cite[Section 3.2]{So2}\footnote{In \cite{So2}, Song works with manifolds which are \emph{thick-at-infinity}, but such manifolds include the class of compact manifolds with minimal boundary.}, we can run level set flow on $\phi(\Gamma \times [-\delta,\delta])$ to find a compact region $W$ containing $\phi(\Gamma \times [-\delta,\delta])$ where $\partial W$ is (possibly empty) minimal and contracting. Note, by the maximum principle, the only minimal hypersurfaces disjoint from $\Gamma$ inside $\Int(W)$ are of the form $\Gamma_t$ for some $t \in (-\delta,\delta)$.
	
	First, we will show that there exists a $t_0 \in [-\delta,0)$ such that any stable minimal hypersurface $\Sigma$ which intersects $\Gamma$ must then also intersect $\Gamma_{t_0}$. Otherwise, we can find a sequence of stable minimal hypersurfaces $\Sigma_k$ intersecting $\Gamma$ but disjoint from $\Gamma_{t_k}$ for some negative sequence $t_k \to 0$. This implies that $\Sigma_k$ converges to $\Gamma$ by the maximum principle. However, since $\Sigma_k$ intersects $\Gamma$ for all $k$, we can construct a Jacobi field for $\Gamma$ that changes sign which contradicts the fact that $\Gamma$ is stable.
	
	So let $U$ be the connected component containing $\Gamma$ of the metric completion of $W \setminus \Gamma_{s_0}$ where $s_0 \in (t_0,0)$ such that $\Gamma_{s_0}$ is minimal and contracting inside $\phi(\Gamma \times [s_0,0])$. Note $U$ is a compact region with nonempty minimal contracting boundary. Furthermore, by construction, $U$ contains no stable minimal hypersurfaces which intersect $\Gamma$. This implies---by a standard area minimization argument---that any two minimal hypersurfaces in $U$ which intersect $\Gamma$ must also intersect each other.
	
	Now although this $(U,g)$ is not a weak core, we will see that from the point of view of min-max, can think of it as being so. By \cite[Lemma 4]{So}, we can find a sequence of metrics $h^{(i)}$ converging smoothly to $g$ such that with respect to $h^{(i)}$,
	\begin{itemize}
		\item $g \equiv h^{(i)}$ on $U \setminus \phi(\Gamma \times [s_0,\delta])$,
		\item $\partial U$ is still minimal and contracting,
		\item $\Gamma = \Gamma_0$ is an unstable minimal hypersurface, and
		\item $\Gamma_t$ has nonzero mean curvature pointing away from $\Gamma$ for $t \in (s_0,0) \cup (0,\delta)$.
	\end{itemize} 
	
	The maximum principle implies that for all $i$, any minimal hypersurface in $(\Int(U),h^{(i)})$ must intersect $\Gamma$. So by Proposition \ref{prop:Song_minmax}, for each $p$ and $i$, we can find a $h^{(i)}$-stationary integral varifold $V_p^{(i)}$ contained in $\Int(U)$ with
	\[
	\widetilde{\omega}_p(U,h^{(i)}) = \|V_p^{(i)}\|(U)
	\] 
	and whose support is a smooth minimal hypersurface with index at most $p$ such that if the multiplicity of any component is more than 1, then that component is stable. For fixed $p$, by \cite[Theorem A.6]{Sh}, after picking a subsequence, $V_p^{(i)}$ converges to some $g$-stationary integral varifold $V_p$ with smooth minimal support and 
	\[
	\|V_p\|(U) = \lim_{i \to \infty} \|V_p^{(i)}\|(U) = \lim_{i \to \infty} \widetilde{\omega}_p(U,h^{(i)}) = \widetilde{\omega}_p(U,g)
	\]
	because the widths vary continuously in the metric. Since each component of $V_p^{(i)}$ intersects $\Gamma$, then so does each component of $V_p$ which implies that the support of $V_p$ is connected. Moreover, $V_p$ must either be an integer multiple $m_p$ of $\Gamma$ or is unstable and hence has multiplicity $m_p = 1$ by \cite[Theorem A.6]{Sh}. 
	
	Now we conclude that there must be minimal hypersurfaces of arbitrarily large area by the same argument in Theorem \ref{thm:WC->large_area}. Again, by \cite[Theorem 9]{So}, there exists some constant $B$ (which is independent of $p$) such that
	\[
	Ap \leq \widetilde{\omega}_p(U,g) \leq Ap+Bp^{\frac{1}{n+1}} \qquad \text{and} \qquad \widetilde{\omega}_{p+1}(U,g) \geq \widetilde{\omega}_p(U,g) + A
	\]
	for all $p$. By the previous part of the proof,  $\widetilde{\omega}_p(U,g)$ is either realized by $m_p \area_g(\Gamma)$ or $\area(\Sigma_p)$ for some unstable minimal hypersurface $\Sigma_p$. By \cite[Lemma 14]{So}, $\widetilde{\omega}_p(U,g)=m_p \area_g(\Gamma)$ can not hold for all sufficiently large $p$. Thus,  there must be minimal hypersurfaces of arbitrarily large area.
\end{proof}

\subsection{Proof of main theorem}
As mentioned, the minimal hypersurfaces found by Lemma \ref{NoncontractingAccumulating} may be very pathological as in Definition \ref{def:non-monotonic}. Note that non-monotonic minimal hypersurfaces are necessarily accumulating on a side. Note such minimal hypersurfaces can possibly exist. See Appendix \ref{sec:example} for an example. So given a Song region $U$, consider
\[
\cM^N(U) = \{\Sigma \in \cM^A(U) : \text{$\Sigma$ is non-monotonic}\},
\]
and let $\cM^N$ be the set of all non-monotonic mininal hypersurfaces in $(M,g)$.

\begin{proposition} \label{prop:nospindle}
	Let $(M,g)$ be as above, and suppose there are no spindle minimal hypersurfaces which are accumulating on a side. Then the space $\cM^N$ of all non-monotonic minimal hypersurfaces is homeomorphic to the Cantor set.
\end{proposition}
\begin{proof}
	First, we will show that $\cM^N$ is nonempty. In fact, we will show that $\cM^N(U)$ is non-empty for all Song regions $U$. Consider a minimal hypersurface $\Gamma$ given by Lemma \ref{NoncontractingAccumulating} so that $\Gamma$ is accumulating on a side but not contracting on the other. 
	
	If the area function of the accumulating side of $\Gamma$ is not weakly monotone, then $\Gamma$ is non-monotonic. Now we consider the case that the area function of the accumulating side of $\Gamma$ is weakly monotone. Note that from the construction in Lemma \ref{NoncontractingAccumulating}, $\Gamma$ has largest area among all the nearby minimal hypersurfaces that is contracting on a side, the area function of the accumulating side of $\Gamma$ is weakly increasing as approaching $\Gamma$. 
	
	If the other side of $\Gamma$ is expanding, then we get a spindle, which is a contradiction. If the other side of $\Gamma$ is accumulating, by the same reason as above, $\Gamma$ is either non-monotonic or the area function of the the otherg side of $\Gamma$ is weakly increasing as approaching $\Gamma$, and hence we get a spindle, which is also a contradiction. So the only possible neighborhood for the other side is locally foliated by minimal hypersurfaces.
	
	If we consider the longest minimal foliation $\Gamma_t$ for $t \in [0,\varepsilon]$ generated by $\Gamma$, then $\Gamma_\varepsilon$ must be non-monotonic. In fact, the other side of $\Gamma_\varepsilon$ can not be minimal foliation, otherwise contradicts the longest minimal foliation assumption. The other side of $\Gamma_\varepsilon$ can not be expanding, otherwise $\Gamma$ would be a spindle minimal hypersurface. The other side of $\Gamma_\varepsilon$ can not be contracting, otherwise $\Gamma_\varepsilon$ is a minimal hypersurface contracting on a side with the same area as $\Gamma$, contradicts the construction in Lemma \ref{NoncontractingAccumulating}.
	
	The only possibility remains is that the other side of  $\Gamma_\varepsilon$ is accumulating. Then we show $\Gamma_\varepsilon$ is non-monotonic by contradiction. Suppose not, by the same reason as we discussed about $\Gamma$ above, the area function of the accumulating side of $\Gamma_\varepsilon$ is weakly increasing as approaching $\Gamma_\varepsilon$. Then we get a spindle again, which is a contradiction. In summary, either $\Gamma$ itself is non-monotonic, or we can find $\Gamma_\varepsilon$ that is non-monotonic.
	
	We will now show that $\cM^N$ is a perfect set. Let $\Gamma \in \cM^N$. Then there exists $\phi: \Gamma \times [0,\delta] \to M$ onto a closed halved tubular neighborhood $\varphi(\Gamma \times [0,\delta])$ where $\Gamma_0 = \Gamma$ and there exists sequences $0 < s_k < t_k$ with $t_k \to 0$ such that $U_k = \varphi(\Gamma \times [s_k,t_k])$ is a Song region. By the above, there exists $\Gamma_k \in \cM^N(U) \subseteq \cM^N$. Since $s_k,t_k \to 0$, we must have that $\Gamma_k$ converges to $\Gamma$ by the maximum principle.
	
	Finally, note that $\cM^N$ is a metric space (from the flat metric) and is totally disconnected. Therefore, by Brouwer's characterization \cite{Br}, $\cM^N$ is homeomorphic to a Cantor set.
\end{proof}

\begin{proof}[Proof of Theorem \ref{thm:main-full}]
	Let $(M^{n+1},g)$ be an arbitrary closed Riemannian manifold with $3 \leq n+1 \leq 7$. In the case, where $(M,g)$ is weakly Frankel, we can find connected minimal hypersurfaces with arbitrarily large area by Theorem \ref{thm:weaklyFlarge}.
	
	If $(M,g)$ is not weakly Frankel, the by Lemma \ref{lem:CutAlongContracting}, we can find Song regions $U$ of $(M,g)$. If one of these Song regions is a weak core, then we can find connected minimal hypersurfaces with arbitrarily large area by Theorem \ref{thm:WC->large_area}.
	
	Finally, assume that $(M,g)$ is not weakly Frankel, does not contain a weak core, and that the area of stable minimal hypersurfaces is uniformly bounded. If $(M,g)$ contains a spindle minimal hypersurface, then we can find connected minimal hypersurfaces with arbitrarily large area by Proposition \ref{prop:spindle}. Otherwise, by Proposition \ref{prop:nospindle}, we are in the second case of Theorem \ref{thm:main-full}. 
\end{proof}

	\section{One-sided minimal hypersurfaces}\label{one-sided case}
	
	In the previous sections, we assumed that $(M,g)$ contains no one-sided minimal hypersurfaces for simplicity. However, our results still hold without this assumption. The purpose of this section is to explain some of the technical modifications needed to handle when $(M,g)$ possibly contains one-sided minimal hypersurfaces.
	
	So suppose $\Gamma \subset (M,g)$ is a one-sided minimal hypersurface. We say that $\Gamma$ is \emph{contracting} (note there is only one side here) if the two-sided double cover is contracting. Equivalently, this says that if $N$ is the metric completion of $M \setminus \Gamma$, then $N$ has a boundary component $\Sigma$ which is contracting and is isometric to the double cover. Also, note if this double cover of $\Gamma$ is merely stable, our Jacobi field arguments used throughout still apply by lifting things to this double cover. 
	
	Finally, note that Zhou \cite{Zh} showed min-max theory does not produce one-sided minimal hypersurfaces for bumpy metrics. In particular, this implies for a general metric, if a one-sided component does appear from min-max, then it will have even multiplicity. Moreover, when the multiplicity is strictly larger than 2, then the double cover must be degenerate stable.
	
	\subsection{Weakly Frankel manifolds}
	Now we discuss the specific modifications needed in Section 3. The definition for weakly Frankel remains the same by using the notion of contracting above. Then all the results follow by considering the double cover whenever working with a one-sided minimal hypersurface. In particular, we can still define when disjoint minimal hypersurfaces $\Gamma_0, \Gamma_1$ (each possibly one-sided) are connected by a minimal foliation by cutting along them and using Lemma $\ref{lem:MinimizeAreaHomology}$. Again, there are three different types of one-sided minimal hypersurfaces $\Gamma \subset (M,g)$:
	
	\begin{enumerate}
		\item We say $\Gamma$ is \emph{isolated} it has no local minimal foliation.
		\item We say $\Gamma$ generates a \emph{partial minimal foliation} if it is connected to other hypersurfaces by minimal foliations where we can parameterize them all as $\Sigma_t = \varphi(\Sigma \times \{t\})$ for some $\varphi: \Sigma \times [0,1] \to M$ where $\varphi|_{\Sigma \times (0,1]}$ is a diffeomorphism onto its image where $\varphi|_{\Sigma \times \{0\}}$ is a double covering onto $\Gamma$.
		\item We say $\Sigma$ generates a \emph{(full) minimal foliation} if it is connected to other hypersurfaces by minimal foliations which union to $M$ and where we can parameterize them all as $\Sigma_t = \varphi(\Sigma \times \{t\})$ for some $\varphi: \Sigma \times [0,1] \to M$ where $\varphi|_{\Sigma \times (0,1)}$ is a diffeomorphism onto its image and where both $\varphi|_{\Sigma \times \{0\}}$ and $\varphi|_{\Sigma \times \{1\}}$ are double covering maps onto one-sided minimal hypersurfaces.
	\end{enumerate}
	We can still consider the cycles defined in Section \ref{SS:Space of stable minimal cycles} associated to the above:
	\begin{enumerate}
		\item We get the zero cycle in $\cZ_n(M;\Z_2)$ since $\Gamma$ occurs with even multiplicity.
		\item Here the space of cycles $\cK_\omega$ is similar as in Lemma \ref{lem:PartialFoliationSpace}, but it will be connected in this case and strongly deformation retracts to the zero cycle (which represents the double cover $\Sigma_0$ of $\Gamma$).
		\item Here the space of cycles $\cF_\omega$ is similar as in Lemma \ref{lem:FullFoliationSpace}. From the above, we get map $S^1 \to \cZ_n(M;\Z_2)$ parametrizing the foliation by considering $\Sigma_t$ as cycles and identifying the double covers $\Sigma_0,\Sigma_1$ with the zero cycle. Again, by taking products of this map and quotienting, we obtain a homeomorphism $TP^m(S^1) \to \cF_\omega$  where now $m$ is the largest integer with $2m \leq \omega/\area_{g}(\Gamma)$. Thus, $\cF_\omega \cong \RP^m$.
	\end{enumerate}
	In particular, we still can describe the topology of the space of all stable cycles $\cS_\omega \subset \cZ_n(M;\Z_2)$ to show that $H^m(\cS_\omega,\Z_2) = 0$ for $m \geq C' \omega$. Then the rest of the arguments follow directly.

	\subsection{Weak core manifolds}
	As before, the one-sided components which appear from Song's min-max occur with even multiplicity, and all our Jacobi field arguments used still apply by lifting things to double covers if necessary. Although, all the results we use from \cite{So} are stated for two-sided hypersurfaces, in \cite{So}, Song handles the one-sided cases (with appropriate modifications). So everything still follows in this case.

	\subsection{Accumulating case}
	Again, we can still define when a one-sided minimal hypersurface $\Gamma \subset (M,g)$ is accumulating or non-monotonic, by considering the two-sided double cover (or equivalently, by cutting $M$ along $\Gamma$ and considering the metric completion). Also, there is no issue in defining one-sided spindles.

	\section{Applications to special manifolds}
	\label{S:Applications to special manifolds}
	\subsection{Foliated by minimal hypersurfaces}
	A special case of a weakly Frankel manifold when the whole manifold is foliated by closed embedded minimal hypersurfaces. In general, even without the weakly Frankel property, the minimal hypersurfaces satisfy some nice properties.
	
	 Note here the foliation is given in the sense of Lemma \ref{lem:MinimizeAreaHomology}: if we cut a manifold $(M,g)$ along a minimal hypersurface $\Sigma$ to get a manifold with boundary $(N,g)$, then there is a diffeomorphism $\varphi:\Sigma\times[0,1]\to N$ such that $\varphi(\Sigma\times\{t\})$ is a minimal hypersurface for all $t\in[0,1]$. In particular, this allows one to construct a fiber bundle $\pi:M \to S^1$ where the fibers $\pi^{-1}(\{\theta\}) = \Sigma_\theta$ parameterize this foliation.
	
	\begin{lemma} \label{lem:minimalfoliations}
		Suppose $(M,g)$ is foliated by closed embedded minimal hypersurfaces. Then any minimal hypersurface $\Gamma$ must either be a leaf of some foliation or intersects every leaf of any minimal foliation. Moreover, when $\Gamma$ is not a leaf of some foliation, the fundamental group of $\Gamma$ is infinite. 
	\end{lemma}
	\begin{proof}
		Fix a foliation of $(M,g)$ by viewing $M$ as a fiber bundle $\pi:M \to S^1$ where the fibers $\pi^{-1}(\{\theta\}) = \Sigma_\theta$ are connected minimal hypersurfaces parameterizing the foliation. By unraveling the base circle, there exists a Riemannian cover $p:\Sigma\times\R \to M$ where each slice $\Sigma\times\{t\}$ is closed minimal hypersurface which projects isometrically to some $\Sigma_\theta$ in our foliation.
		
		Suppose $\Gamma$ is a minimal hypersurface in $(M,g)$ which does not intersect every leaf $\Sigma_\theta$ in the foliation. Then any loop contained in $\Gamma$ can be homotoped inside $M$ to a loop contained in some fixed slice $\Sigma_{\theta_{0}}$. By the lifting criterion, we can lift to obtain a minimal hypersurface $\hat{\Gamma} \subset \Sigma\times\R$ which projects isometrically to $\Gamma$. Since $\hat{\Gamma}$ is compact, it must touch some minimal leaf $\Sigma \times \{t\} \subset \Sigma\times\R$ from one side. Therefore, by the maximum principle, $\hat{\Gamma}$ must equal some slice $\Sigma \times \{t\}$, and hence, $\Gamma$ must equal some leaf $\Sigma_\theta$.
		
		Finally, suppose $\Gamma$ is a minimal hypersurface which is not equal to some leaf of the foliation. We can lift the universal cover $\tilde{\Gamma}$ of $\Gamma$ to get a minimal immersion $\tilde{\Gamma} \to \Sigma \times \R$. But if $\pi_1(\Gamma)$ were finite, then $\tilde{\Gamma}$ is also compact which would again give contradiction by the maximum principle.
	\end{proof}

	\begin{remark}
		The second part of the lemma can be made stronger. For instance, if $\Gamma$ is a minimal hypersurface which is not equal to some leaf of the foliation, one can show $\pi_* i_*(\pi_1(\Gamma)) \subseteq \pi_1(S^1)$ must be infinite where $i: \Gamma \to M$ is the inclusion.
	\end{remark}
	
	\subsection{Analytic manifolds}
	The {\L}ojasiewicz-Simon inequality is a powerful tool to study the local behavior of a critical point of an analytic elliptic integrand functional, see \cite{Si}. In the special case that the analytic functional is chosen to be the area functional in an analytic manifold, the result implies that:
	
	\begin{theorem}[{\L}ojasiewicz-Simon inequality]
		Suppose $(M,g)$ is an analytic manifold and $\Sigma$ is a minimal hypersurface. If a sequence of minimal hypersurfaces $\{\Sigma_i\}$ converges to $\Sigma$ smoothly, then when $i$ is sufficiently large, $\area(\Sigma_i)=\area(\Sigma)$.
	\end{theorem}

	However, we only need to know this for stable minimal hypersurfaces, so we give an explicit proof this fact in Corollary \ref{cor:nonisolatedanalytic} using the implicit function theorem. In particular, we show if $\Sigma$ is a non-isolated stable minimal hypersurface, then $\Sigma$ must locally be a minimal foliation on both sides\footnote{In fact, compactness and the maximum principle tell us that this local foliation extends to a (full) minimal foliation of $(M,g)$.}, As a consequence, we have that:
	
	\begin{corollary} \label{cor:analyticnoaccumlating}
		In an analytic metric, there exist no minimal hypersurfaces which are accumulating on a side.
	\end{corollary}

	Note the same holds for bumpy metrics simply because accumulating minimal hypersurfaces are degenerate. However, unlike bumpy metrics, analytic metrics can (and often will) have minimal foliations. Such foliations can complicate the space of minimal hypersurfaces, but as shown in Section \ref{WeaklyFrankelCase}, we can control such things.

	\begin{proof}[Proof of Theorem \ref{thm:analytic}]
		Let $(M^{n+1},g)$ with $3 \leq n+1 \leq 7$ have an analytic metric. Note that in Section 5, we showed if a manifold is not weakly Frankel, has no weak core, and has a uniform area bound for stable minimal hypersurfaces, then there must exist some accumulating minimal hypersurface (for example, see Lemma \ref{NoncontractingAccumulating}). But by Corollary \ref{cor:nonisolatedanalytic}, there exist no accumulating minimal hypersurfaces in analytic metrics. 
		
		Therefore, either $(M,g)$ must either have stable minimal hypersurfaces of arbitrarily large area, be weakly Frankel, or has a weak core. In the later cases, we can find arbitrarily large minimal hypersurfaces by Theorem \ref{thm:weaklyFlarge} and Theorem \ref{thm:WC->large_area}.
	\end{proof}
	
	\subsection{Manifolds with positive scalar curvature}
	
	In a $3$-manifold with positive scalar curvature, the area of a minimal surface is controlled by the index. This is a consequence of the following theorem proved by Chodosh-Ketover-Maximo:
	\begin{theorem}[\cite{CKM} Theorem 1.3] 
		Suppose $(M^3,g)$ is a closed $3$-manifold with positive scalar curvature. For any $I \in \N$, there exist $A_0=A_0(M,g,I) > 0$ such that if $\Sigma$ is a closed embedded minimal surface in $(M,g)$, then
		\[
		\area(\Sigma) \leq A_0 \quad \text{whenever} \quad \ind_g(\Sigma) \leq I.
		\]
	\end{theorem}
	
	Now we are willing to show Corollary \ref{cor:analytic+PSC} and Corollary \ref{cor:foliation+PSC}. 
	
	\begin{proof}[Proof of Corollary \ref{cor:analytic+PSC}]
		Suppose $(M^3,g)$ is analytic with positive scalar curvature. Since the metric is analytic, we can find minimal surfaces of arbitrarily large area by Theorem \ref{thm:analytic}. Therefore, the result follows from \cite[Theorem 1.3]{CKM}.
	\end{proof}
	
	\begin{proof}[Proof of Corollary \ref{cor:foliation+PSC}]
		Suppose $(M^3,g)$ has positive scalar curvature and admits a foliation by closed embedded minimal surfaces. We will show that $(M,g)$ must be weakly Frankel. Recall that Lemma \ref{lem:minimalfoliations} shows that any minimal surface is either a leaf in some minimal foliation, or intersects each leaf of every minimal foliation. Because the manifold $(M,g)$ has positive scalar curvature, a classical result of Schoen-Yau \cite{SY} shows that any stable minimal surfaces must be topologically either a sphere or real projective plane (in particular, has finite fundamental group). Thus, the second part of Lemma \ref{lem:minimalfoliations} shows that any stable minimal surface in $(M,g)$ must be a leaf of a foliation. Therefore, $(M^3,g)$ must be weakly Frankel, and hence has minimal hypersurfaces of arbitrarily large area by Theorem \ref{thm:weaklyFlarge}.
	\end{proof}

	\appendix
	
	\section{Nice neighborhood lemma} \label{sec:NiceNeighborhood}
	
	\begin{lemma} \label{lem:NiceNeighorhood}
		Let $\Gamma$ be a closed two-sided minimal hypersurface in $(M,g)$. There exists a foliation $\{\Gamma_t\}_{t \in [-\delta,\delta]}$ of some tubular neighborhood $N$ such that $\Gamma_0 = \Gamma$ and where for each fixed $t \in [-\delta,\delta]$, the mean curvature of $\Gamma_t$ is either entirely zero, positive, or negative. Moreover, the foliation is parameterized by a diffeomorphism
		\[
		\Gamma_t = \varphi(\Gamma \times \{t\}) \quad \text{where} \quad \varphi: \Gamma \times [-\delta,\delta] \to N
		\]
		such that $\varphi(x,0) = x$ for all $x \in \Gamma$.
	\end{lemma}
		\begin{proof}
		For $\phi \in C^\infty(\Gamma)$, consider the smooth hypersurfaces given by
		\[
		\Gamma_\phi = \{\exp_x(\phi(x)\nu(x)) : x \in \Gamma\}.
		\]
		Pick a neighborhood $U$ around $0 \in C^\infty(\Gamma)$ such that $\Gamma_\phi$ is embedded for every $\phi \in U$. Consider the smooth map $H:U \to C^\infty(\Gamma)$ where $H(\phi)$ is the mean curvature of $\Gamma_\phi$ (pulled back to be a function on $\Gamma$). The differential at $0 \in C^\infty(\Gamma)$
		\[
		DH_0: C^\infty(\Gamma) \to C^\infty(\Gamma)
		\]
		is given by the Jacobi operator of $\Gamma$, that is, $DH_0 = L_\Gamma$ where
		\[
		L_\Gamma = -\Delta-|A|^2-\Ric_M(\nu,\nu).
		\]
		Let $\lambda$ be the least eigenvalue for $L_\Gamma$. Recall the corresponding eigenspace for $\lambda$ is spanned by a single eigenfunction $\phi_0$ which we can assume to be positive. Note
		\[
		H(t\phi_0) = H(0) + DH_0(t\phi_0) + O(t^2) = tL_\Gamma(\phi_0) + O(t^2) = t \lambda \phi_0 + O(t^2)
		\]
		by Taylor expansion. Therefore, if $\lambda \neq 0$, then for small enough $\delta > 0$,
		\[
		\varphi: \Gamma \times [-\delta,\delta] \to M \quad \text{given by} \quad \varphi(x,t) = \exp_x(t\phi_0(x)\nu(x))
		\]
		is the desired local foliation. Moreover, we indeed see that when $\Gamma$ is unstable (equivalent to $\lambda < 0$), this foliation is expanding on both sides. Likewise, when $\Gamma$ is strictly stable (equivalent to $\lambda > 0$), this foliation is contracting on both sides.
		
		So now, assume that $\Gamma$ is degenerate stable, that is, $\lambda = 0$. Note that we have
		\[
		K := \ker(L_\Gamma) = \Span(\phi_0).
		\]
		Let $\pi: C^\infty(\Gamma) \to K^\perp$ be the projection onto the $L^2$ orthogonal complement $K^\perp \subset C^\infty(\Gamma)$ of the kernel, and then consider the map $H^\perp: C^\infty(\Gamma) \to K^\perp$ given by $H^\perp = \pi \circ H$. Decompose the domain as $C^\infty(\Gamma) = K \oplus K^\perp$, and note for $\phi \in K^\perp$,
		\[
		DH^\perp_{(0,0)}(0,\phi) = D \pi_{(0,0)} (DH_{(0,0)}(0,\phi)) = \pi(L_\Gamma(\phi)) = L_\Gamma(\phi)
		\]
		is invertible as map $K^\perp \to K^\perp$ by the Fredholm alternative, and the inverse is bounded (because the spectrum of $L_\Gamma$ is discrete). Therefore, by the implicit function theorem, there exists an $\varepsilon > 0$ and a neighborhood $W \subset U$ around $0$ along with a map $\Phi:(-\varepsilon,\varepsilon) \to K^\perp$ with $\Phi'(0) = 0$ such that for all $\phi \in W$,
		\[
		H^\perp(\phi) = 0 \quad \text{if and only if} \quad \phi = t\phi_0+\Phi(t)
		\]
		for some $t \in (-\varepsilon,\varepsilon)$. In particular, there exists $c:(-\varepsilon,\varepsilon)\to\R$ such that
		\[
		H(t\phi_0+\Phi(t)) = c(t)\phi_0.
		\]
		Therefore, for $\delta > 0$ sufficiently small, the map
		\[
		\varphi: \Gamma \times [-\delta,\delta] \to M \quad \text{given by} \quad \varphi(x,t) = \exp_x((t\phi_0(x)+\Phi(t)(x))\nu(x))
		\]
		gives the desired local foliation.
		\end{proof}
	
		\begin{corollary} \label{cor:nonisolatedanalytic}
			If $\Gamma$ is a two-sided non-isolated stable minimal hypersurface in an analytic metric $(M,g)$, then the local foliation given by Lemma \ref{lem:NiceNeighorhood}
			\[
			\Gamma_t = \varphi(\Gamma \times \{t\}) \quad \text{where} \quad \varphi: \Gamma \times [-\delta,\delta] \to N
			\]
			is a minimal foliation, that is, $\Gamma_t$ is minimal for all $t \in [-\delta,\delta]$.
		\end{corollary}
		\begin{proof}
			Since $\Gamma$ is assumed to be non-isolated, $\Gamma$ is degenerate stable, that is, we are in the case where $\lambda = 0$ from the previous proof. In that notation, for all $\phi \in W$,
			\[
			H(\phi) = 0 \quad \text{if and only if} \quad \phi = t\phi_0+\Phi(t) \quad \text{and} \quad c(t) = 0
			\]
			for some $t \in (-\delta,\delta)$. Since the metric is analytic, the map $H:W \to C^\infty(\Gamma)$ is analytic, and thus, the map $c:(-\delta,\delta) \to \R$ is analytic as well. Since $\Gamma$ is non-isolated, the function $c(t)$ has an accumulation of zeros. Therefore, $c = 0$ identically by analyticity, and so, $\Gamma_t$ is minimal for all $t \in (-\delta,\delta)$.
		\end{proof}
	
	\begin{lemma}\label{lem:accumulating-nearby}
		Suppose $\Gamma$ is a minimal hypersurface in $(M,g)$ and the halved tubular neighborhood $\varphi(\Gamma\times[0,\delta])$ is accumulating, then for any $\delta'\in(0,\delta)$, there exists $s\in(0,\delta')$, such that $\Gamma_s:=\varphi(\Gamma\times\{s\})$ is contracting at a side.
	\end{lemma}

\begin{proof}
	Write $f(t)=\area(\varphi(\Sigma\times\{t\}))$. Corollary \ref{cor:nonisolatedanalytic} implies that we only need to show for any $\delta'\in(0,\delta)$, there exists $s\in(0,\delta')$, such that $f'(s)=0$, and there exists $\epsilon>0$ such that either $f'(t)<0$ for $t\in(s-\epsilon,s)$, or $f'(t)>0$ for $t\in(s,s+\epsilon)$.
	
	Because $\Gamma$ has halved neighborhood that is accumulating, for any $\delta'>0$, we can always find $0<s_1<t_0<s_2<\delta'$, such that $f'(s_1)=f'(s_2)=0$, and $f'(t_0)\neq 0$. If $f'(t_0)<0$, we choose $s=\sup_{r}\{r>t_0:f'(t)<0, t\in[t_0,r)\}$; if $f'(t_0)>0$, we choose $s=\sup_{r}\{r<t_0:f'(t)>0, t\in(r,t_0]\}$. In either cases, we have $f'(s)=0$, and either $f'(t)<0$ for $t\in(t_0,s)$, or $f'(t)>0$ for $t\in(s,t_0)$. This concludes the proof.
\end{proof}
	
	\section{Pathological example} \label{sec:example}
	
	We give a class of examples of smooth Riemannian manifold where the second case happens from Theorem \ref{thm:main} and Theorem \ref{thm:main-full}.
	
	\begin{proposition} \label{Athm:example}
		For $n\geq 2$, let $(\Sigma^n,g_0)$ be any closed Riemannian manifold. There exists a smooth metric $(\Sigma \times S^1,g)$ such that the space $\cM^N$ of non-monotonic minimal hypersurfaces is homeomorphic to the Cantor set $C$.
	\end{proposition}

	We will construct a warped product metric using the pathological function:

	\begin{lemma} \label{Alem:function}
	There exists a smooth positive periodic function $f:\R \to \R$ where:
	\begin{enumerate}
		\item each critical point is either non-isolated or a strict local minimum;
		\item for each non-isolated critical point $p \in \R$, we have that $f(x)$ is not weakly monotone on at least one side $[p-\varepsilon,p]$ or $[p,p+\varepsilon]$ for every $\varepsilon > 0$.
	\end{enumerate}
	\end{lemma}
	\begin{proof}
		Recall the standard Cantor set construction where we start with $C_0 = [0,1]$ and where given $C_{n-1}$ which consists of $2^{n-1}$ disjoint closed intervals centered at the points $m_{n,1},m_{n,2} \dots, m_{n,2^{n-1}}$, then we form $C_n$ by removing open intervals of length $1/3^n$ centered about those midpoints. Then the Cantor set $C$ is defined to be the intersection of all $C_n$.
		
		Let $\Psi: \R \to \R$ be the smooth bump function given by $\exp(1/(x^2-1))$ on $(-1,1)$ and zero elsewhere, and let $\Psi_{n,k}: \R \to \R$ be this map translated and rescaled as
		\[
			\Psi_{n,k}(x) = \Psi(2 \cdot 3^n(x-m_{n,k}))
		\]
		so that $\Psi_{n,k}$ is non-zero exactly on the open middle third centered at $m_{n,k}$. Define
		\[
		h(x) = \sum_{m \in \Z}^\infty \sum_{n=1}^\infty \sum_{k=1}^{2^{n-1}} c_n \Psi_{n,k}(x-m)
		\]
		and pick the coefficients $c_n$ to decay quick enough to make the function smooth and satisfy $h(x) < 1$. We define $f:\R \to (0,\infty)$ by $f(x)=1-h(x)$, and we claim that such an $f$ satisfies item (2) and (3). Since the function has period 1, it suffices to look just at $f$ on $[0,1]$.
		
		To prove item (2), notice that the critical points $p \in [0,1]$ are precisely either in $C$ or is some middle third midpoint $m_{n,k}$. Note that $p \in C$ must be non-isolated because $C$ is a perfect set and every point in $C$ is a critical point. And, if $p$ is some midpoint $m_{n,k}$, then $p$ is a strict local minimum because $m_{n,k}$ is the strict maximum of $\Psi_{n,k}$.
		
		Next, we prove item (3). From the above, the set $C$ is all of the non-isolated critical points in $[0,1]$. By construction, for all $\varepsilon > 0$, we can find some midpoint such that $0<|p-m_{n,k}|<\varepsilon$. Note that $f$ is not weakly monotone at $m_{n,k}$ Since $m_{n,k}$ is the strict maximum of $\Psi_{n,k}$.
	\end{proof}
	
	\begin{proof}[Proof of Proposition \ref{Athm:example}]
		We consider the smooth warped product metric 
		\[
		g = f(t)^2g_0 + dt^2 \quad \text{on} \quad \Sigma \times S^1
		\]
		where we are identifying $S^1$ here as $\R/\Z$ or just $[0,1]$ with the ends identified. Consider the slices $\Sigma_t = \Sigma \times \{t\}$. Note $\Sigma_t$ is a minimal hypersurface of $(\Sigma \times S^1,g)$ if and only if $t \in [0,1]$ is a critical point of $f$. Moreover, the minimal slices $\Sigma_t$ are non-monotonic if and only if $t \in [0,1]$ is a non-isolated critical point of $f$ which is given by the Cantor set by Lemma \ref{Alem:function}.
	\end{proof}
	
	\begin{remark}
		It is not known whether or not such a metric admits minimal hypersurfaces with arbitrarily large area.
	\end{remark}

\medskip

\subsection*{Acknowledgments} The authors want to thank Professor Andre Neves for his suggestion on this problem and valuable discussions. The second author wants to thank Xin Zhou for helpful conversations.

\end{document}